\documentclass[12pt]{amsart}
\usepackage{mathtools}
\usepackage{pdflscape}
\usepackage{amsfonts}
\usepackage{amsmath}
\usepackage{cases}
\usepackage{amsfonts}
\usepackage{amssymb}
\usepackage{amssymb}
\usepackage{amssymb}
\usepackage[margin=1in]{geometry}
\usepackage{mathrsfs}
\usepackage{amsmath}
\usepackage{verbatim}
\usepackage{hyperref}
\usepackage{amsfonts}
\usepackage{bbm}
\usepackage{young}
\usepackage{amscd}
\usepackage{color}
\usepackage{amssymb}
\usepackage[all, knot]{xy}
\usepackage[x-cp1252]{inputenx}
\usepackage{cite}
\xyoption{arc} 

\usepackage{datetime}
\usepackage{amstext}
\usepackage{amsxtra}
\usepackage{amsopn}
\usepackage{amsthm}
\newtheorem{thm}{Theorem}[section]
\theoremstyle{definition}
\newtheorem{defn}{Definition}[section]
\theoremstyle{plain}

\theoremstyle{remark}
\newtheorem{rem}{Remark}[section]
\theoremstyle{plain}
\newtheorem{lem}[thm]{Lemma}
\theoremstyle{plain}
\newtheorem{cor}[thm]{Corollary}
\theoremstyle{plain}

\theoremstyle{conjecture}

\usepackage{cancel}

\newcommand{\Sym}{{\rm Sym}}

\newcommand{\mc}{\mathcal}
\CompileMatrices
\usepackage{listings}
\begin{document}
\title{Classical sheaf cohomology rings on Grassmannians}
\author{Jirui Guo}
\address{Physics Department, Robeson Hall (0435), Virginia Tech, Blacksburg, VA  24061, USA }
\email{jrkwok@vt.edu}
\author{Zhentao Lu}
\address{Mathematical Institute, University of Oxford, Andrew Wiles Building, Oxford OX2 6GG, UK}
\email{zhentao@sas.upenn.edu}
\author{Eric Sharpe}
\address{Physics Department, Robeson Hall (0435), Virginia Tech, Blacksburg, VA  24061, USA }
\email{ersharpe@vt.edu}

\date{\today}

\begin{abstract}
Let the vector bundle $\mc{E}$ be a deformation of the tangent bundle
over the Grassmannian $G(k,n)$. We compute the ring structure of
sheaf cohomology valued in exterior powers of $\mc{E}$,
also known as the
polymology. This is the first part of a project studying the
quantum sheaf cohomology of Grassmannians with deformations of the
tangent bundle, a generalization of ordinary quantum cohomology rings
of Grassmannians.
A companion physics paper \cite{Guo2017}
describes physical aspects of the theory,
including a conjecture for the quantum sheaf cohomology ring,
and numerous examples.
\end{abstract}

\maketitle

\section*{Introduction}

Let $X$ be a compact K\"ahler manifold and $\mc{E}$ be a holomorphic vector
bundle on $X$ which satisfies $c_i(\mc{E})=c_i(T_X)$, $i=1,2$.
Such bundles are sometimes called \textit{omalous bundles}.
There is a ring structure
on $\oplus_{p,q}H^q(X,\wedge^p\mc{E}^*)$ called the polymology,
see Section \ref{sec:polymology}. The study of the polymology, and its
quantum corrections, is a relatively new theory known as
quantum sheaf cohomology (QSC), which generalizes the ordinary quantum
cohomology of a space.

QSC was first described in \cite{katz2006notes},
and the mathematical theory of QSC was first worked out for deformations of
tangent bundles on toric varieties in
\cite{donagi2011physical,donagi2014mathematical}, based on physics results
in \cite{mcorist2009summing}
(see the companion paper \cite{Guo2017} and survey papers
\cite{mcorist2011revival, melnikov2012recent} for more physics background).
Briefly, the quantum corrections to the ring structure are computed by
applying sheaf cohomology to induced sheaves on a moduli space of curves,
rather than intersection theory as is the norm for ordinary quantum
cohomology.
When one takes $\mc{E}=T_X$, QSC reduces to the usual quantum cohomology
of $X$.
In general, QSC, as well as related correlation functions
(see for example \cite{lu2015correlator,mcorist2008half,mcorist2009summing}),
provide new invariants of omalous bundles.

As a step towards understanding QSC for Grassmannians,
in this paper we derive the classical sheaf cohomology ring
(polymology) for Grassmannians with vector bundles given by
deformations of the tangent bundle.
The companion paper
\cite{Guo2017} gives physics results for both classical and quantum
sheaf cohomology rings for such cases.  Mathematically rigorous derivations of
the QSC rings on Grassmannians, checking the physics results in
\cite{Guo2017},
are left
for future work.

The content of this paper is as follows. In Section \ref{sec:polymology} 
we define the polymology ring and the deformed tangent bundles of 
Grassmannians, whose cohomology is the main object of this paper. 
In Section \ref{sec:VBcohomologies} we introduce the notations for the 
homogeneous vector bundles canonically constructed by Weyl modules and 
Schur functors. We then describe a version of Borel-Weil-Bott theorem to 
compute the cohomology of homogeneous vector bundles on Grassmannians and 
present concrete results. These results will later be used to compute the 
cohomology of the deformed tangent bundles. In 
Section \ref{sec:degenerate_locus} we work out the degenerate locus where the 
map $f\in Hom(\mc{S\otimes S^*}, \mc{V\otimes S^*})$ fails to define a 
deformed tangent bundle. 
In the process we parametrize the deformations of $f$
in terms of an $n \times n$
matrix $B$.
In Section \ref{sec:moduli} we show that the deformations considered in this paper covers all isomorphism classes of first order deformations of the cotangent bundle. We also confirm that the deformed tangent bundles, under a precise condition, are not isomorphic to the tangent bundle. In Section \ref{sec:quotient_commutativity}, we show that the polymology is a quotient ring for generic deformed tangent bundles. In Section \ref{sec:cohomology_computation}, we perform the main computations. We first work out some general results about the $B$-dependence of the cohomology $H^r(\wedge^r\mc{E}^*)$, for the deformed tangent bundle $\mc{E}$ defined by $f=f_B$. Then we focus on the $r=n-k+1$ case and show that the result for  $G(k,n)$ shares the same form as the result for $G(k+c,n+c)$ (See Theorem \ref{thm:kernel_independent_of_n}). Next we compute the special case $B = \varepsilon I$, where $I$ is the identity matrix. This result then helps to determine the result for generic $\mc{E}$, stated as Theorem \ref{result}.
In Section \ref{sec:non-generic} and \ref{sec:qsc}, we discuss the non-generic 
situation and the conjecture for the quantum sheaf cohomology briefly, 
drawing the reader's attention to the examples of the former and the 
discussions of the latter in the companion physics paper 
\cite{Guo2017}. 
We conclude the paper in Section \ref{sec:conclusion}. 
Appendix A contains a technical result about Cech cohomology representatives.

\section{The classical cohomology ring}\label{sec:polymology}

We work over the complex numbers. In view of the equivalence of the category of algebraic vector bundles and that of locally free sheaves and the GAGA principle, we will constantly switch points of view and regard $\mc{E}$ as a holomorphic vector bundle, or an algebraic one, or the sheaf of holomorphic / algebraic sections of the vector bundle in this paper.

In general, we can use Cech cohomology to define a cup product,
just as in \cite{donagi2014mathematical}.

\begin{defn}
Let $X$ be a smooth projective complex algebraic variety and $\mc{E}$ be a vector bundle over $X$. The \textit{polymology} of $\mc{E}$ is the classical sheaf cohomology ring defined as $\bigoplus_{p,q} H^q(\wedge^p\mc{E}^*)$ with the multiplication (cup product)
\begin{equation}
H^q(\wedge^p \mc{E}^*) \times  H^{q^\prime}(\wedge^{p^\prime} \mc{E}^*)\to H^{q+q^\prime}(\wedge^{p+p^\prime}\mc{E}^*),
\end{equation}
defined by the natural maps
\begin{equation}
 H^q(\wedge^p \mc{E}^*) \times  H^{q^\prime}(\wedge^{p^\prime} \mc{E}^*)\to  H^q(\wedge^p \mc{E}^*) \otimes  H^{q^\prime}(\wedge^{p^\prime} \mc{E}^*) \to H^{q+q^\prime}(\wedge^p\mc{E}^*\otimes\wedge^{p^\prime} \mc{E}^*)
 \end{equation}
in Cech cohomology, followed by the map induced from the projection $\wedge^p\mc{E}^*\otimes\wedge^{p^\prime} \mc{E}^*\to\wedge^{p+p^\prime}\mc{E}^*$. We denote it by $H^*_\mc{E}(X)$.
\end{defn}

This is analogous to the product structure discussed in Chapter 14 of
\cite{MR658304}.

We first define our notation.
Let $V$ be an $n$-dimensional complex vector space and $X=G(k,V)$ be the
Grassmannian of $k$-planes in $V$.
In this paper we will assume $1 < k < n-1$, {\it i.e.} $X$ is not the projective space.
Let $\mc{S}$ be the tautological subbundle, $\mc{V}$ be the trivial bundle $X\times V$, and $\mc{Q}$ be the quotient bundle. They fit in the short exact sequence
\begin{equation}
0\to \mc{S}\to\mc{V}\to\mc{Q}\to 0.
\end{equation}
The tangent bundle $T_X \cong \mc{Q}\otimes\mc{S}^*$ then is the cokernel
\begin{equation}\label{seq:ses_tangent}
0 \to \mc{S}\otimes\mc{S}^* \to \mc{V}\otimes\mc{S}^* \to  T_X \to 0.
\end{equation}

In this paper we focus on rank $k(n-k)$ bundles defined by the short exact sequence
\begin{equation}\label{seq:define_E}
0 \to \mc{S}\otimes\mc{S}^* \xrightarrow{g} \mc{V}\otimes\mc{S}^* \to  \mc{E} \to 0.
\end{equation}

We will refer to these bundles simply as \textit{deformed tangent bundles} and their dual vector bundles as \textit{deformed cotangent bundles}.
We show in 
Section \ref{sec:moduli} that they covers all infinitesimal  deformation directions of the cotangent bundle.

When $\mc{E} \cong TX$, the sheaf cohomology groups above are ultimately
dual to Schubert subvarieties, and the quantum corrections to the product
lead to quantum cohomology.  For a deformation, the sheaf cohomology groups
above have a physical interpretation as `gauge-invariant operators' in
a gauged linear sigma model, and the quantum corrections to the
product lead to QSC.

\section{The cohomology of homogeneous bundles}\label{sec:VBcohomologies}
\subsection{Weyl modules and Schur functors}

To compute the polymology, we use the dual sequence of \eqref{seq:define_E},
\begin{equation}
 0 \to \mc{E}^* \to \mc{V^*\otimes S} \xrightarrow{f} \mc{S^*\otimes S}\to 0,
\end{equation}
and its Koszul resolutions (see Section \ref{sec:quotient_commutativity}), which involve homogeneous vector bundles. As $GL_n$ modules, the homogeneous bundles decompose into Weyl modules indexed by Young diagrams, e.g. $K_\lambda \mc{V}^*$ and $K_\beta \mc{S}^* \otimes K_\gamma \mc{Q}^*$. See Chapter 6 of
 \cite{MR1153249} for background.

Let $\lambda^\prime$ be the transpose of the Young diagram $\lambda$. Since
we work in characteristic zero, Weyl modules and Schur functors are directly related:
\begin{equation}
K_\lambda M \cong L_{\lambda^\prime} M.
\end{equation}

Schur functors can also be applied to complexes (See Chapter 2 of\cite{weyman2003cohomology}). We first state a result for Schur complexes, generalizing the familiar Koszul complexes.
\begin{thm}[Chapter 2, Exercise 21 of \cite{weyman2003cohomology}]\label{schur complex}
For a short exact sequence of $\mathbb{C}$-vector spaces
\begin{equation}
0\to F_1 \xrightarrow{\Psi} F_0 \to M\to 0
\end{equation}
and a Young diagram $\lambda^\prime$ of weight $r$ (namely $\sum \lambda_i =r$), we have a $(r+2)$-term long exact sequence
\begin{equation}\label{concrete}
0\to K_{\lambda^\prime} F_1 \to ...\to K_\lambda F_0 \to K_\lambda M\to 0.
\end{equation}
More precisely, define the Schur complex $L_{\lambda^\prime}\mathbb{E}$ for the complex $\mathbb{E}: F_1\to F_0$ as
\begin{equation}
(L_{\lambda^\prime}\mathbb{E})_r \to ... \to (L_{\lambda^\prime}\mathbb{E})_1 \to (L_{\lambda^\prime}\mathbb{E})_0,
\end{equation}
with \begin{equation}\label{Schur_complex_component_as_direct_sum}
(L_{\lambda^\prime}\mathbb{E})_t = \bigoplus_{|\nu|=r-t} K_{\lambda^\prime/\nu} F_1 \otimes K_{\nu^\prime} F_0.
\end{equation}
Then (\ref{concrete}) is exactly
\begin{equation}
0 \to (L_{\lambda^\prime}\mathbb{E})_r \to ... \to (L_{\lambda^\prime}\mathbb{E})_1 \to (L_{\lambda^\prime}\mathbb{E})_0 \to K_\lambda M \to 0.
\end{equation}\end{thm}

Note that we quote the result with $\lambda^\prime$ instead of $\lambda$,
to better mesh with our notation.
Also, we are working on vector spaces, so it is automatically a $(r-1)^{st}$ syzygy module for any $r\geq 1$. The direct sum decomposition (\ref{Schur_complex_component_as_direct_sum}) is the characteristic zero case of (2.4.10), part (a) of \cite{weyman2003cohomology}, see (2.3.1) of \cite{weyman2003cohomology}.

Also, note that
\begin{equation}\label{skew_schur}
 K_{\lambda^\prime/\nu} F = \displaystyle\bigoplus_{|\mu|=|\lambda^\prime| -|\nu|} c^{\lambda^\prime}_{\mu\nu}K_\mu F,
 \end{equation}
 where $c^{\lambda^\prime}_{\mu\nu}$ is the Littlewood-Richardson coefficient\footnote{See for example p83 of \cite{MR1153249}, or (2.3.6) of Weyman\cite{weyman2003cohomology}.}.
\qed

Now we apply Theorem \ref{schur complex} to the short exact sequence
\begin{equation}
0\to \mathcal{Q}^*\to \mathcal{V}^*\to \mathcal{S}^*\to 0
\end{equation}
to get
\begin{equation}\label{K_lambda}
0\to K_{\lambda^\prime} \mathcal{Q}^* \to ... \to K_\lambda \mathcal{V}^*\to K_\lambda \mathcal{S}^*\to 0.
\end{equation}
Tensoring this sequence with $K_\lambda\mathcal{S}$, we get
\begin{equation}\label{E lambda}
E_\lambda:
0\to K_{\lambda^\prime} \mathcal{Q}^* \otimes K_\lambda\mathcal{S} \to ... \to K_\lambda \mathcal{V}^*\otimes K_\lambda\mathcal{S}
\to K_\lambda \mathcal{S}^*\otimes K_\lambda\mathcal{S}
\to 0.
\end{equation}

The maps in this sequence are naturally induced from the tautological sequence, hence the map
\begin{equation}\label{delta_lambda}
\delta_{\lambda}^r  : H^0( K_\lambda \mathcal{S}^*\otimes K_\lambda\mathcal{S})
\to
H^r( K_{\lambda^\prime} \mathcal{Q}^* \otimes K_\lambda\mathcal{S})
\end{equation}
 on cohomology is also naturally induced from it.

Taking $\lambda = (r)$, we also have the following induced sequence from the dual of \eqref{seq:ses_tangent},
\begin{equation}\label{E r}
E^r: 0\to \wedge^r \Omega \to ... \to \Sym^r(\mathcal{V}^*\otimes \mathcal{S})
\to \Sym^r (\mathcal{S}^*\otimes \mathcal{S})
\to 0,
\end{equation}
with induced map
\begin{equation}
\delta^r: H^0(\Sym^r (\mathcal{S}^*\otimes\mathcal{S}) ) \to H^r (\Omega^r).
\end{equation}

Comparing (\ref{E lambda}) and (\ref{E r}), we have the following theorem:
\begin{thm}\label{thm:direct_sum}
The complex (\ref{E r}) factorizes as
\begin{equation}
E^r = \oplus_{\lambda\in \mathcal{P}(k,r)} E_\lambda,
\end{equation}
and
\begin{equation}
\delta^r = \oplus_{\lambda\in \mathcal{P}(k,r)} \delta_\lambda^r,
\end{equation}
\end{thm}
\begin{proof}
The proof presented here is computational and relies on facts
that are only true in characteristic zero.

Comparing terms, one finds that it suffices to show that
\begin{equation}\label{single_term}
\bigoplus_{|\lambda|=r} \bigoplus_{|\nu|=r-t} K_{\lambda^\prime/\nu} \mathcal{Q}^* \otimes K_{\nu^\prime} \mathcal{V}^* \otimes K_\lambda\mathcal{S}
=
\wedge^t (\mathcal{Q}^*\otimes\mathcal{S}) \otimes \Sym^{r-t}(\mathcal{V}^*\otimes\mathcal{S}).
\end{equation}

Note that we have
\begin{equation}
\wedge^t (\mathcal{Q}^*\otimes\mathcal{S}) = \bigoplus_{|\mu|=t} K_{\mu}\mathcal{Q}^*\otimes K_{\mu^\prime}\mathcal{S}
\end{equation}
and
\begin{equation}
\Sym^{r-t}(\mathcal{V}^*\otimes\mathcal{S})
= \bigoplus_{|\nu|=r-t} K_{\nu^\prime} \mathcal{V}^*\otimes K_{\nu^\prime}\mathcal{S},
\end{equation}
where we write $\nu^\prime$ (which is the transpose of the Young diagram $\nu$) instead of $\nu$ purely for the convenience of manipulating notations.

Now we apply (\ref{skew_schur}) to the LHS of (\ref{single_term}), and get
\begin{equation}\label{long_derivation}
\begin{array}{ll}
{\rm LHS} &=\displaystyle \bigoplus_{|\lambda|=r} \bigoplus_{|\nu|=r-t} K_{\lambda^\prime/\nu} \mathcal{Q}^* \otimes K_{\nu^\prime} \mathcal{V}^* \otimes K_\lambda\mathcal{S} ,\\
& =\displaystyle \bigoplus_{|\lambda|=r} \bigoplus_{|\nu|=r-t} 
\left(\bigoplus_{|\mu|=t} c^{\lambda^\prime}_{\mu\nu}K_\mu \mathcal{Q}^* \right)
 \otimes K_{\nu^\prime} \mathcal{V}^* \otimes K_\lambda\mathcal{S},\\
& =\displaystyle \bigoplus_{|\lambda|=r} \bigoplus_{|\nu|=r-t} \left(
\bigoplus_{|\mu|=t} c^{\lambda}_{\mu^\prime\nu^\prime}K_\mu \mathcal{Q}^*\right)
 \otimes K_{\nu^\prime} \mathcal{V}^* \otimes K_\lambda\mathcal{S},\\
& =\displaystyle \bigoplus_{|\nu|=r-t} \bigoplus_{|\mu|=t} K_\mu \mathcal{Q}^* \otimes K_{\nu^\prime} \mathcal{V}^* \otimes
\left(\bigoplus_{|\lambda|=r} c^{\lambda}_{\mu^\prime\nu^\prime} K_\lambda\mathcal{S}\right),\\
& =\displaystyle \bigoplus_{|\nu|=r-t} \bigoplus_{|\mu|=t} K_\mu \mathcal{Q}^* \otimes K_{\nu^\prime} \mathcal{V}^* \otimes(K_{\mu^\prime}S\otimes K_{\nu^\prime}S),\\
& =\displaystyle \left(\bigoplus_{|\mu|=t} 
K_\mu \mathcal{Q}^*\otimes K_{\mu^\prime}S \right)  \otimes 
\left(
\bigoplus_{|\nu|=r-t} K_{\nu^\prime} \mathcal{V}^* \otimes K_{\nu^\prime}S
\right),\\
& = {\rm RHS},
\end{array}
\end{equation}
where we used the property $c^{\lambda}_{\mu\nu} =c^{\lambda^\prime}_{\mu^\prime\nu^\prime}$ for Littlewood-Richardson coefficients (Corollary 2, Section 5.1 of \cite{fulton1997young}).
\end{proof}

\begin{rem}\label{rem1}
Applying Theorem \ref{schur complex} to \begin{equation*}
0 \to\mc{ S \to V \to Q }\to 0,
\end{equation*} we get an exact sequence
\begin{equation*}
0 \to K_\lambda \mc{S} \to ... \to K_{\lambda^\prime} \mc{V} \to K_{\lambda^\prime}  \mc{Q} \to 0,
\end{equation*}
and its dual,
\begin{equation}
0  \to K_{\lambda^\prime} \mc{Q^*} \to K_{\lambda^\prime}  \mc{V^*} \to ...\to K_\lambda \mc{S^*} \to 0.
\end{equation}
Tensoring it with $K_\lambda \mc{S}$, and sum over $|\lambda|=r$, we get a decomposition of
\begin{equation}\label{the_true_LEG}
\begin{array}{ll}
0 &\to \Omega^r \to \wedge^r(\mc{V}^*\otimes \mc{S}) \to ...\\
&\to (\mc{V}^*\otimes \mc{S})\otimes \Sym^{r-1}(\mc{S^*\otimes S}) \to  \Sym^r(\mc{S^*\otimes S})\to 0.
\end{array}
\end{equation}
This can be proved analogously to Theorem \ref{thm:direct_sum}, by verifying
\begin{equation}
\displaystyle\bigoplus_{|\lambda|=r}\bigoplus_{|\nu|=t} K_{\lambda/\nu}\mc{S}^*\otimes K_{\nu^\prime}\mc{V}^* \otimes K_{\lambda}\mc{S} = \wedge^t(\mc{V}^*\otimes \mc{S})
\otimes \Sym^{r-t}(\mc{S^*\otimes S}).
\end{equation}
\end{rem}

\subsection{Borel-Weil-Bott theorem for homogeneous bundles on Grassmannians}

To compute the cohomologies, we quote a version of the Borel-Weil-Bott Theorem from \cite{weyman2003cohomology} (Corollary 4.1.9).
\begin{thm}\label{BWB}(Borel-Weil-Bott)
For each vector bundle of the form $K_{\beta}\mathcal{S}^*\otimes K_{\gamma}\mathcal{Q}^*$, where $K_\beta\mathcal{S}^* $ and $K_{\gamma}\mathcal{Q}^*$ are Weyl modules, the only non-vanishing cohomology lives in dimension $l(\alpha)$, if there is a way transfering $\alpha = (\beta,\gamma)$ into a dominant weight $\tilde{\alpha}$ of $GL(n)$ and $l(\alpha)$ is the number of elementary transformations performed.

In this case, we have
\begin{equation}
H^{l(\alpha)}(K_{\beta}\mathcal{S}^*\otimes K_{\gamma}\mathcal{Q}^*) = K_{\tilde{\alpha}}{V}^*.
\end{equation}
\end{thm}
The elementary transformations will be simply called \textit{mutations} in this paper. They are easily described in concrete terms.
Specifically, a mutation maps $\alpha = (\alpha_1,...,\alpha_n)$
with $\alpha_i < \alpha_{i+1}$,
to $(\alpha_1,...,\alpha_{i+1} - 1, \alpha_{i}+1,...,\alpha_n)$.

\subsection{The cohomology of homogeneous bundles}

We work out some results of the cohomology of homogeneous bundles on Grassmannians for later use.

Let $\mathcal{P}(k, r)$ be the set of all partitions of $r$ with at most $k$ parts.
As a $GL_n$ representation, the zeroth isotypical  component of $H^0(\Sym^r(\mathcal{S}^*\otimes\mathcal{S}))$ is:

\begin{equation}\label{sym_r}
\begin{array}{ll}
H^0_0(\Sym^r(\mathcal{S}^*\otimes\mathcal{S}))
& \cong \bigoplus_{\lambda\in \mathcal{P}(k,r)} H^0_0 (  K_\lambda \mathcal{S}^* \otimes K_\lambda \mathcal{S}),\\
& \cong \bigoplus_{\lambda\in \mathcal{P}(k,r)}\mathbb{C}\cdot \kappa_{\lambda},
\end{array}
\end{equation}
where $K_\lambda \mathcal{S}$ is the Schur functor associated to the Young diagram $\lambda=(\lambda_1,...,\lambda_k)$, with $\lambda_i$ being the number of boxes in the $i$-th row.

Furthermore, we can show that $H^0(\Sym^r(\mc{S}^*\otimes \mc{S}))$ is a
trivial
$GL(V)$-module\footnote{This is easy to prove from Borel-Weil-Bott:
the components $K_\lambda \mc{S}^*$ of $\Sym^r(\mc{S}^*\otimes \mc{S})$ satisfy
$|\lambda|=0$. If $\lambda\neq 0$, then $\lambda_k <0$. Hence
$(\lambda_1,...,\lambda_k,0,...,0)$ is not non-increasing and there is no
contribution to $H^0$.}.
Hence we have
\begin{equation}\label{sym_r2}
\begin{array}{ll}
H^0(\Sym^r(\mathcal{S}^*\otimes\mathcal{S}))
& \cong \bigoplus_{\lambda\in \mathcal{P}(k,r)} H^0(  K_\lambda \mathcal{S}^* \otimes K_\lambda \mathcal{S}),\\
& \cong \bigoplus_{\lambda\in \mathcal{P}(k,r)}\mathbb{C}\cdot \kappa_{\lambda}.
\end{array}
\end{equation}

We can also compute the polymology for the cotangent bundle.
First note that
\begin{equation}
\Omega^r \cong \wedge^r (\mathcal{Q}^*\otimes\mathcal{S})
\cong \oplus_{\lambda\in\mathcal{P}(k,r)} K_{\lambda^\prime} \mathcal{Q}^*\otimes K_{\lambda}\mathcal{S},
\end{equation}
where $\lambda^\prime$ is the transpose of the Young diagram $\lambda$.

Of course, when $\lambda_1 > n-k$, $\lambda^\prime$ has more than $n-k$ rows,
and so $K_{\lambda^\prime}\mathcal{Q}^*$ vanishes.

\begin{thm}\label{nonvanishing}
When $\lambda\in \mathcal{P}(k,r)$ satisfies $\lambda_1\leq n-k$ (or pictorially the Young diagram $\lambda$ is contained in the $(k\times (n-k))$ rectangle), we have
\begin{equation}
H^r( K_{\lambda^\prime} \mathcal{Q}^*\otimes K_{\lambda}\mathcal{S} ) \cong \mathbb{C},
\end{equation}
and the rest $H^j$'s are 0.
\end{thm}
\begin{proof}
in order to apply Theorem \ref{BWB}, we need to write
\begin{equation}
K_{\lambda^\prime} \mathcal{Q}^*\otimes K_{\lambda}\mathcal{S} \cong
K_{\bar{\lambda}} \mathcal{S}^*\otimes K_{\lambda^\prime}\mathcal{Q}^*,
\end{equation}
where $\bar{\lambda} = (-\lambda_k,...,-\lambda_1)$ when $\lambda=(\lambda_1,...,\lambda_k)$.

Now we only need to mutate $(-\lambda_k,...,-\lambda_1,\lambda^\prime_1,...,\lambda^\prime_{n-k})$.

We claim that
we can mutate $(-\lambda_k,...,-\lambda_1,\lambda^\prime_1,...,\lambda^\prime_{n-k})$ into $(0^{n})$, and the number of steps is $r$.

By Theorem \ref{BWB}, this claim implies
\begin{equation}
H^r( K_{\lambda^\prime} \mathcal{Q}^*\otimes K_{\lambda}\mathcal{S} )
\cong K_{(0^n)} V^* \cong \mathbb{C},
\end{equation}
which proves our theorem.

Proof of Claim:

We simply carry out the mutations. It is easily done pictorially. We write down $-\lambda_j$ and draw boxes in columns representing $\lambda_i^\prime$, such that the number of boxes in Column $i$ is exactly $\lambda_i^\prime$:

$$\left(-\lambda_k,...,-\lambda_2,-\lambda_1,
\begin{Young}
&&&&&&      \cr
&&&&    \cr
&&\cr
\end{Young}
\right).
$$

However if we look at the rows of the diagram, row $j$ has exactly
$\lambda_j$ boxes by the fact that $\lambda$ is the transpose of
$\lambda^\prime$.

Hence, we can mutate $\lambda_1$ times  to get
$$
\left(-\lambda_k,...,-\lambda_2,
\begin{Young}
&&&&    \cr
&&\cr
\end{Young},
0 \right),
$$
where the diagram is simply the Young diagram of $(\lambda_2,...,\lambda_k)$ when look at the rows. The $0$ at the end is the mutation result of $-\lambda_{1}$.

Repeating the procedure we get each row annihilated with a $-\lambda_j$, resulting in $(0^{n})$ in $\sum_j \lambda_j =r$ steps.
\end{proof}

Here is a vanishing result for the cohomology of $K_\lambda\mc{S}\otimes K_\mu \mc{Q}^*$.
\begin{thm}\label{vanishing}
On the Grassmannian $G(k,n)$, if $\lambda\in \mc{P}(k,r), \mu \in \mc{P}(n-k,s)$ with $\lambda_1\leq n-k$, then a sufficient condition for $H^{\bullet}(K_\lambda \mc{S}\otimes K_\mu \mc{Q}^*)$ to vanish is that $\mu_j^{\prime} < \lambda_j$ for some $j$.
\end{thm}
Proof: The proof is similar to that of Theorem \ref{nonvanishing}. It boils down to mutating $\tau=(-\lambda_k,...,-\lambda_1, \mu_1,...,\mu_{n-k})$.

If $\mu^{\prime}_1 < \lambda_1$, then one can mutate $\tau_k = -\lambda_1$ to the right for $\lambda_1$ times and get a Young diagram $\tau^{(\lambda_1)}$.
 Note that then $\tau^{(\lambda_1)}$ will satisfy that $\tau^{(\lambda_1)}_{k+\lambda_1-1} = -1$ and $\tau^{(\lambda_1)}_{k+\lambda_1} = 0$. This $(...,-1,0,...)$ shows that $\tau$ cannot be mutated to a decreasing sequence, hence all cohomology vanishes.

If $\mu^{\prime}_i \geq \lambda_i$, $i=1,2,...,j-1$, and $\mu^{\prime}_j < \lambda_j$, then one can perform the above mutations and get $\tau^{(\lambda_1)} = (-\lambda_k,...,-\lambda_2,\mu_1-1,...,\mu_{\lambda_1}-1,0,...)$, and further get $\tau^{(\lambda_1,...,\lambda_{j-1})} = (-\lambda_k,...,-\lambda_j,\mu_1 - (j-1),...)$. Then one runs into the same situation as the $\mu_1^{\prime} < \lambda_1$ case and concludes that the cohomology vanishes.\qed
\begin{cor}
On the Grassmannian $G(k,n)$, if $\lambda\in\mc{P}(k,r), \lambda_1 \leq n-k, \mu^\prime \subsetneq \lambda$, then
\begin{equation}
H^{\bullet}(K_\lambda \mc{S}\otimes K_\mu \mc{Q}^*) =0.
\end{equation}

\end{cor}
This shows the vanishing of the intermediate cohomologies of (\ref{E r}), since one has $\mu^\prime \subsetneq \lambda$ whenever $t\neq0$ and $t\neq r$ in the third line of (\ref{long_derivation}). Hence we know the map $\delta_\lambda^r$ in (\ref{delta_lambda}) is an isomorphism for $\lambda$ with $\lambda_1\leq n-k$.

\begin{cor}
If $\alpha\in\mc{P}(k,r), \alpha_1\leq n-k, \beta\in\mc{P}(n-k,r)$ are two Young diagrams with the same weight $r$, and $\beta^\prime \neq \alpha$ then
\begin{equation}
H^{\bullet}(K_\alpha \mc{S}\otimes K_\beta \mc{Q}^*) =0.
\end{equation}
\end{cor}
This is clear since $|\alpha|=|\beta|$ and $\beta^\prime \neq \alpha$ implies there exists $j$ such that $\beta^\prime_j < \alpha_j$.

\begin{rem}
A condition (also in \cite{weyman2003cohomology}) equivalent to the vanishing of $H^\bullet(K_\lambda \mc{S}^*\otimes K_\mu\mc{Q}^*)$ is that there exists $\sigma\in \Sigma_n$ (the symmetric group on $n$ letters) $\sigma\cdot\alpha = \alpha$, for $\alpha=(\lambda,\mu)$, where $\sigma\cdot\alpha = \sigma(\alpha+\rho) -\rho$, and $\rho = (n-1, n-2, ...,0)$. Equivalently, this requires $\alpha+\rho$ has repetitive entries.

Let $\nu=(\nu_1,...,\nu_k)$ with $|\nu|=\sum \nu_i<0$. $H^i(K_\nu \mc{S}^*)$ vanishes for all $i$ iff $(\nu_1+n-1,\nu_2+n-2,..,\nu_k+n-k,n-k-1,n-k-2,...,1,0)$ has repetitive entries. Since $\nu_j\geq \nu_{j+1}$, the condition reduces to the existence of at least one $j\in\{1,...,k\}$ such that $0\leq v_j+n-j\leq n-k-1$. In particular, $-(n-k)\leq \nu_k\leq -1$ suffices.
\end{rem}
In particular, we can get

\begin{cor}\label{cor:K_lambda_S}
For $\lambda>0$, namely $\lambda_j > 0$, $\forall j$, we have $\bigoplus_i H^i(K_\lambda \mc{S}) \neq 0$ iff $\exists j$ such that $\lambda_j\geq n-k+j > j \geq \lambda_{j+1}$. Moreover, when this condition holds, we have \begin{equation}
H^{j(n-k)}(K_\lambda \mc{S}) = K_{(-\lambda_k,...,-\lambda_{j+1},-j,...,-j,-\lambda_j+(n-k),...,-\lambda_1+(n-k))}V^*.
\end{equation}
In particular,
$
H^m(K_\lambda \mc{S}) = 0,
$
when $|\lambda|=m$.
\end{cor}

\begin{thm}\label{vanishing_r_leq_n-k}
For each $\lambda\in \mc{P}(k,r)$ with $\lambda_1\leq n-k$, if $0\subsetneq\nu\subseteq \lambda$, then \begin{equation}
H^\bullet(K_{\lambda/\nu}\mc{S}^*\otimes K_\lambda \mc{S}) = 0.
\end{equation}
\end{thm}
\begin{proof}
Since $K_{\lambda/\nu}\mc{S}^* =\bigoplus_{\mu}c^{\lambda}_{\mu\nu}K_{\mu}\mc{S}^*$, it suffices to prove that for each $\mu$ such that $0\subseteq\mu\subsetneq \lambda$,
$H^\bullet(K_{\mu}\mc{S}^*\otimes K_\lambda \mc{S}) =0$.

To do this, we denote $K_{\mu}\mc{S}^*\otimes K_\lambda \mc{S}$ as $\bigoplus_\beta K_\beta \mc{S}^*$. Note that $K_\lambda \mc{S} = K_{(\lambda_1 - \lambda_k,...,\lambda_1-\lambda_2,0)}\mc{S}^*\otimes (\wedge^k\mc{S})^{\lambda_1}$, and for each component $K_\alpha\mc{S}^*$ in $K_\mu \mc{S}^*\otimes K_{(\lambda_1 - \lambda_k,...,\lambda_1-\lambda_2,0)}\mc{S}^* =\bigoplus_\alpha K_\alpha \mc{S}^*$, we have $0\leq \alpha_k < \lambda_1$ from the fact that $|\alpha| =|\mu|+|(\lambda_1 - \lambda_k,...,\lambda_1-\lambda_2,0)|$. Since $\beta_k =\alpha_k -\lambda_1$, this implies that $-\lambda_1 \leq \beta_k < 0$, hence $H^\bullet(K_\beta \mc{S}^*)=0$ by the above remark.
\end{proof}

\section{The degenerate locus}\label{sec:degenerate_locus}

By (\ref{seq:define_E}), the deformed tangent bundle $\mc{E}$ is determined by the map $f\in Hom(\mc{S\otimes S^*}, \mc{V\otimes S^*})$. Using the results in Section \ref{sec:VBcohomologies}, we find
\begin{equation}
\begin{array}{ll}
Hom(\mc{S\otimes S^*}, \mc{V\otimes S^*}) 
 & \cong  H^0(\mc{S^*\otimes S}\otimes \mc{V\otimes S^*}), \\
& \cong  H^0(\mc{S^*\otimes S \otimes S^*})\otimes \mc{V}, \\
& \cong  H^0(K_{(2,0,...,0,-1)} \mc{S^*} \oplus K_{(1,1,...,0,-1)} \mc{S^*} \oplus \mc{S^*} \oplus \mc{S^*} )\otimes \mc{V}, \\
& \cong  (0\oplus 0\oplus \mc{V}^*\oplus \mc{V}^*)\otimes \mc{V}, \\
& \cong  \mc{V}^*\otimes \mc{V} \oplus \mc{V}^*\otimes \mc{V} .
\end{array}
\end{equation}

Note that here we used our assumption that $k > 1$.

The map $f$ can be written down explicitly. Locally we have
\begin{equation}\label{f_for_tangent}
f: \lambda \mapsto  \lambda_a^b A^i_j \phi^j_b  + (tr \lambda) B^i_j \phi^j_a,
\end{equation}
where $a, b$ are $S$ indices, and $i, j$ are $V$ indices. 
When ($A^i_j$), $i,j=1,...,n$ form an invertible matrix, 
we can always set $A^i_j =\delta^i_j$ (the Kronecker delta), 
using the $GL(V)$ action on $\mc{V}$. 
So it remains to consider the $B$-deformations. 
We will write $f$ as $f_B$ to indicate the $B$-dependence and view $B$ as a $n\times n $ matrix.

The \textit{degenerate locus} of $B$-deformations is the set of $B$ such that the cokernel of $f_B$ fails to be a deformed tangent bundle.

In this section we work out the degenerate locus.
\begin{lem}\label{lem:invariant_subspaces}
Let $\mathbb{B}$ be a linear operator acting on an $n$-dimensional vector space $V$. Then
for any $k$ eigenvalues (counting multiplicity) $\lambda_{1},...,\lambda_{k}$ of $\mathbb{B}$, one can always find a $k$ dimensional invariant subspace $V_k\subset V$ such that $\lambda_{1},...,\lambda_{k}$ are the eigenvalues of $\mathbb{B}|_{V_k}$.
\end{lem}
The proof is elementary and so omitted.  We note that we are working over
an algebraically closed field. 
%

Dual to (\ref{f_for_tangent}), $f_B: \mc{V^*\otimes S}\to \mc{S^*\otimes S}$ can be written as
\begin{equation}\label{map:f_B}
 f_B: c^a_i \mapsto c^a_i v^i_b  + c^d_i B^i_j v^j_d \delta^a_b.
 \end{equation} For the kernel to be a deformed cotangent bundle, we need to ensure the map is of rank $k^2$ at every point of the Grassmannian $G(k,n)$.

Denote the image of $f_B$ as a tuple $(\sigma^a_b)_{a,b=1,...,k}$, then
\begin{equation}
\sigma^a_b =  c^{a^\prime}_{i^\prime} (\delta^a_{a^\prime} v^{i^\prime}_b +\delta^a_b B^{i^\prime}_j v^j_{a^\prime}).
\end{equation}
So $f_B$ is represented by a big $k^2 \times kn$ matrix $M$. We use $(a,b)$ as the row index of $M$, and $(a^\prime, i^\prime)$ as the column index.

Write $M$ as $M_1+M_2$, where
$M_1 = \rm{diag}\{\mathbf{V,...,V}
\}
$ with $\mathbf{V}_{b,i^\prime} =v^{i^\prime}_b$, corresponding to $\delta^a_{a^\prime} v^{i^\prime}_b $, and $M_2$ has non-vanishing rows only when $a=b$, and each such row has entry $B^{i^\prime}_j v^j_{a^\prime}$ at place $(a^\prime, i^\prime)$.

Now we want to know the equivalent condition for $\rm{rank}(M)<k^2$.

For the case $k=1$, this is equivalent to $B^i_j v^j_1 + v^i_1 =0, \forall i$. In matrix language, this says there are solutions for $\textbf{V(B+I)}=0$. So the condition is \begin{equation}\label{eigen_require_for_k=1}
\det \textbf{(I+B)} =0.
\end{equation}

When $k\geq 2$,
we first perform a partial Gauss elimination on $M$: for each $b=2,3,...,k$, subtract the first row from row $(b,b)$. The result matrix $M^\prime$ is identical to $M_1$, except the first row and the first $n$ columns.

Note that $\rm{rank}(M)<k^2$ iff the rows of $M^\prime$ are linearly dependent.

Write down the linear-dependence condition $\sum c_{ab} M^\prime_{(a,b)} = 0$, where $M^\prime_{(a,b)}$ is the $(a,b)-th$ row. Observe that the undeformed $B=0$ case implies that we can assume $$c_{11} =1.$$ Then, because of the `almost-diagonal' nature of $M^\prime$, we can spell the conditions out for each column of $M^\prime$, and repackage them into
\begin{equation}\label{cvvb}
\mathbf{CV = VB},
\end{equation}
where we have $$\mathbf{C}_{ab}=\left\{\begin{matrix}
-\sum_{j=1}^k c_{jj}, & a=b=1,
\\
c_{ab}, & otherwise.
\end{matrix}\right.$$ and $\mathbf{B}_{ji^\prime} = B^{i^\prime}_j$.

Hence we conclude
\begin{thm}\label{thmcvvb}
The $B$-deformation fails to define a vector bundle iff there exists at least one point in $G(k,n)$ such that (\ref{cvvb}) has non-zero solutions.
\end{thm}
Note that the constraint on $\mathbf{C}$ is equivalent to $\rm{tr\ }\mathbf{C} =-1$.
It is independent of the choice of the Stiefel coordinates $\mathbf{V}$.
Moreover, it suffices to consider the Jordan canonical form of $B$ since  $\mathbf{CV = VB}$ is equivalent to $\mathbf{CVN = VN N^{-1}BN}$, $\mathbf{N}\in GL(n)$.

\begin{thm}\label{thm:general_degen_loci}
An $n\times n$ matrix $B$ is in the degenerate locus for $G(k,n)$ iff

$\rm{(*)}$ there exists $k$ eigenvalues $\lambda_1,...,\lambda_k$ of $B$ such that $\sum \lambda_i = -1$.
\end{thm}
\begin{proof}The $k=1$ case is done before, since this is equivalent to (\ref{eigen_require_for_k=1}). For $k\geq2$ Theorem \ref{thmcvvb} shows that we need to consider the solutions of $\textbf{CV=VB}$ for each $\textbf{V}$, which is a Stiefel coordinate of the point $[\textbf{V}]\in G(k,n)$.

For each $\mathbf{V}$, we can always find a $g\in GL(V)$ such that $\mathbf{V} = \mathbf{V}_0 g$, where $\mathbf{V}_0 = (\mathbf{I_{k}\ 0})$ when written as a block matrix. Let $\mathbf{\tilde{B}}=\mathbf{B}_g = g\mathbf{B}g^{-1}$ So it suffices to consider $\mathbf{CV_0 =V_0\tilde{B}}$, for all $g\in GL(V)/\mathfrak{B}$, where $\mathfrak{B}$ is the Borel subgroup that leaves $[\mathbf{V}_0]$ fixed.

Observe that $\mathbf{CV_0 =V_0\tilde{B}}$ has a solution with $\rm{tr\ }\mathbf{C}=-1$ is equivalent to $\mathbf{\tilde{B}} =\begin{pmatrix}
\mathbf{J}_{11} & 0\\
\mathbf{J}_{21}& \mathbf{J}_{22}
\end{pmatrix} $ in block matrix notion with $\rm{tr\ }\mathbf{J}_{11} =-1$.

Recall that we have the (strange) notation conversion $B=\mathbf{B}^T$. So we can reformulate the equivalent condition for the $B$-deformation fails to give rise a vector bundle on $G(k,n)$ as

$\rm{(**)}$ \textit{
there exists $g\in GL(V)$ such that $\tilde{B}=B_g =g^{-1} B g =\begin{pmatrix}
J_{11} & J_{12}\\0 & J_{22}
\end{pmatrix}$ with $\rm{tr\ }J_{11}=-1$.
}

View $B$ as the matrix representation of a linear operator $\mathbb{B}$ on $V$ under the standard basis $\{e_1,...,e_n\}$.Then $\tilde{B}$ is the matrix representation of the same linear operator in the new basis $\{\tilde{e}_1,...,\tilde{e}_n\}=\{ge_1,...,ge_n\}$. Also note that $\tilde{B}$ is of the block upper triangular form $\begin{pmatrix}
J_{11} & J_{12}\\0 & J_{22}
\end{pmatrix}$ iff $\mathbb{B}V_k \subset V_k$, where $V_i =\rm{span}\{\tilde{e}_1,...,\tilde{e}_i\}$.

So the problem reduces to the determination of $k$ dimensional invariant subspaces of $V$ under the operator $\mathbb{B}$.

Note that $\mathbb{B}|_{V_k}$ is an linear operator whose eigenvalues are also eigenvalues of $\mathbb{B}$. On the other hand, Lemma \ref{lem:invariant_subspaces} says that
for any $k$ eigenvalues (counting multiplicity) $\lambda_{1},...,\lambda_{k}$ of $\mathbb{B}$, one can always find a $k$ dimensional invariant subspace $V_k\subset V$ such that $\lambda_{1},...,\lambda_{k}$ are the eigenvalues of $\mathbb{B}|_{V_k}$.
 This implies that $\rm{tr\ } J_{11}$ will always be a sum of $k$ eigenvalues of $B$, and any $k$ eigenvalues of $\mathbb{B}$ can be the eigenvalues of $J_{11}$. Hence $\rm{(**)}$ is equivalent to $\rm{(*)}$.\end{proof}
\begin{rem} Unlike results in later sections, this result is true for all $B$-deformations, not just generic deformations.
\end{rem}

\begin{thm}\label{thm:degenerate_locus_in_B}
For $G(k,n)$, the degenerate locus can be described as
\begin{equation}
\displaystyle\det ( \wedge^k I + \sum_{j=0}^{k-1} (\wedge^j I)\wedge B\wedge (\wedge^{k-1-j} I) )=0.
\end{equation}
\end{thm}

In particular, when $k=1,2,3$, the expression is
\begin{equation}
\begin{array}{l}
\det(I+B)=0, \\ \det(I\wedge I + B\wedge I + I \wedge B)=0,\\
\det (I\wedge I\wedge I + B\wedge I\wedge I + I\wedge B\wedge I + I\wedge I \wedge B)=0,
\end{array}
\end{equation}
respectively.
\begin{proof}[Proof of Theorem \ref{thm:degenerate_locus_in_B}]
View $B$ as the matrix representation of a linear operator $\mathbb{B}$ on $V$ under the standard basis $\{e_1,...,e_n\}$. It suffices to prove the case when $B$ is of the Jordan canonical form. Suppose the diagonal elements of $B$ are $\lambda_1,...,\lambda_n$. They are also the eigenvalues of $B$. $\{e_{i_1...i_k}:=e_{i_1}\wedge ... \wedge e_{i_k},i_1 < ... < i_k\}$ is a basis of $\wedge^k V$ and we order the base vectors lexicographically. Note that $Be_i = \lambda_i e_i +\epsilon_i e_{i+1}$, where $\epsilon_i$ is either $0$ or $1$ and $(B\wedge I) (e_{i}\wedge e_j) = \lambda_i e_i\wedge e_j + \delta_i e_{i+1}\wedge e_j$, etc. It is then easy to see that $$\wedge^k I + \sum_{j=0}^{k-1} (\wedge^j I)\wedge B\wedge (\wedge^{k-1-j} I)$$ is an upper-triangular matrix and the diagonal element in the row corresponding to $e_{i_1...i_k}$ is $1+\lambda_{i_1}+...+\lambda_{i_k}$. So the determinant is exactly $\prod (1+\lambda_{i_1}+...+\lambda_{i_k})$.
\end{proof}

\section{The moduli}\label{sec:moduli}

In this section we consider the moduli of the deformation of the (co)tangent bundles. For the infinitesimal moduli, we show that the $B$-deformations represent all the Kodaira-Spencer classes of the cotangent bundle. For the global moduli, we show that $B$-deformations indeed generate vector bundles that are not isomorphic to the (co)tangent bundle, which indicates that in physics applications like \cite{Guo2017}, one gets genuine new physical theories when turning on the $B$-deformations. 

First we compute the Kodaira-Spencer classes of our $B$-deformations for the cotangent bundle.

\begin{thm}
Let $1<k<n-1$ and $X=G(k,n)$ be the Grassmannian. Let $\mathbf{B}$ be the space of $B$ matrices not in the degenerate locus, where the origin $0\in\mathbf{B}$ corresponds to the cotangent bundle. Then the Kodaira-Spencer map $T_0\mathbf{B} \to Ext^1(\Omega,\Omega)$ is surjective.
\end{thm} 

\begin{proof}
We follow Section 1.2, Theorem 2.7 of \cite{hartshorne}. We briefly recall from the theorem that, the infinitesimal deformations of a coherent sheaf $\mc{F}$
over the dual numbers $D=\text{Spec\ } \mathbb{C}[\varepsilon]/\varepsilon^2$ can be represented by short exact sequences in the form
\begin{equation}\label{seq:extension}
0 \to \mc{F} \to \mc{F}^\prime \to \mc{F} \to 0,
\end{equation}
which is derived by applying $\mc{F}^\prime\otimes_{\mc{O}_D}-$ to the short exact sequence 
\begin{equation}
0 \to \mathbb{C} \xrightarrow{\cdot \varepsilon} \mathbb{C}[\varepsilon]/\varepsilon^2 \to \mathbb{C} \to 0.
\end{equation}
Here $\mc{F}^\prime$, a coherent sheaf of $\mc{O}_{X\times D}$ modules, 
can be viewed as a coherent sheaf of $\mc{O}_X$ modules, via a splitting map of $\mathbb{C}[\varepsilon]/\varepsilon^2 \to \mathbb{C}$ (like $\mathbb{C} \to \mathbb{C}[\varepsilon]/\varepsilon^2$, $t\mapsto t$). Now \eqref{seq:extension} defines an extension class in $Ext^1(\mc{F},\mc{F})$. To compute it, apply the functor $Hom(\mc{F},-)$ to \eqref{seq:extension}, and take the derived sequence 
\begin{equation}
0 \to Hom(\mc{F}, \mc{F}) \to Hom(\mc{F},\mc{F}^\prime) \to Hom (\mc{F},\mc{F}) \xrightarrow{\delta} Ext^1(\mc{F}, \mc{F})\to ... .
\end{equation}
Then the Kodaira-Spencer class is the image $\delta([id])$ of the connecting homomorphism $\delta$.

Consider the origin $0\in\mathbf{B}$ and a tangent direction $w$. They determine a morphism $\varphi_{0.w}: D\to \mathbf{B}$ uniquely. Varying $B$ in $\mathbf{B}$, the short exact sequence 
\begin{equation}\label{seq:deformed_omega}
0 \to \mc{E}^* \to \mc{V^*\otimes S} \xrightarrow{f_B} \mc{S^*\otimes S} \to 0 
\end{equation}
forms a short exact sequence over $X\times \mathbf{B}$. Pulling back the latter sequence via the morphism $1\times \varphi_{0,w}: X\times D \to X\times \mathbf{B}$, we have
\begin{equation}
0 \to (\mc{E}^*)^\prime \to (\mc{V^*\otimes S})^\prime \xrightarrow{f_B} (\mc{S^*\otimes S})^\prime \to 0.
\end{equation}
Over $B=0$ the bundle $\mc{E}^*$ is the just the cotangent bundle $\Omega$. 
Since $\mc{V^*\otimes S}$ and $\mc{S^*\otimes S}$ are independent from $B$-deformations, we have $(\mc{V^*\otimes S})^\prime \cong \varepsilon(\mc{V^*\otimes S})\oplus (\mc{V^*\otimes S})$ as $\mc{O}_X$ modules, and similarly $(\mc{V^*\otimes S})^\prime \cong \varepsilon(\mc{S^*\otimes S})\oplus(\mc{S^*\otimes S})$.

In view of \eqref{seq:extension}, we have a diagram of exact rows and columns 
\begin{equation}
\xymatrix{
&0\ar[d] &0\ar[d] &0\ar[d]\\
0 \ar[r] &\Omega \ar[r]\ar[d] & (\mc{E}^*)^\prime \ar[r]\ar[d] & \Omega \ar[r]\ar[d]  &0\\
0 \ar[r] & \mc{V^*\otimes S} \ar[r]\ar[d] & \varepsilon(\mc{V^*\otimes S})\oplus (\mc{V^*\otimes S}) \ar[r]\ar[d] & \mc{V^*\otimes S} \ar[r]\ar[d]  &0\\
0 \ar[r] & \mc{S^*\otimes S} \ar[r]\ar[d] & \varepsilon(\mc{S^*\otimes S})\oplus(\mc{S^*\otimes S}) \ar[r]\ar[d] & \mc{S^*\otimes S} \ar[r]\ar[d]  &0\\
&0 &0 &0
}
\end{equation}
Applying $Hom(\Omega,-)$, we have (next page, 
using that by Borel-Weil-Bott $Ext^1(\Omega,\mc{V^*\otimes S}) =0$)
\begin{landscape}
\begin{equation}\label{diag:hom_omega}
\xymatrix{
&0\ar[d] &0\ar[d] &0\ar[d]\\
0 \ar[r] &Hom(\Omega,\Omega) \ar[r]\ar[d] & Hom(\Omega,(\mc{E}^*)^\prime) \ar[r]\ar[d] & Hom(\Omega,\Omega) \ar[r]^{\delta}\ar[d]  &Ext^1(\Omega,\Omega)\ar[d]\\
0 \ar[r] & Hom(\Omega,\mc{V^*\otimes S}) \ar[r]\ar[d]^{f_{I,0}\circ -} & Hom(\Omega,\varepsilon(\mc{V^*\otimes S})\oplus (\mc{V^*\otimes S})) \ar[r]\ar[d] & Hom(\Omega,\mc{V^*\otimes S}) \ar[r]\ar[d]  &Ext^1(\Omega,\mc{V^*\otimes S})\ar[d]\\
0 \ar[r] & Hom(\Omega,\mc{S^*\otimes S}) \ar[r]\ar[d]^{\delta_2} & Hom(\Omega,\varepsilon(\mc{S^*\otimes S})\oplus(\mc{S^*\otimes S})) \ar[r]\ar[d] & Hom(\Omega,\mc{S^*\otimes S}) \ar[r]\ar[d]  &Ext^1(\Omega,\mc{S^*\otimes S})\\
&Ext^1(\Omega,\Omega)\ar[d] &Ext^1(\Omega,(\mc{E}^*)^\prime) &Ext^1(\Omega,\Omega)\\
&0
}
\end{equation}
\end{landscape}
To find the Kodaira-Spencer class $\delta([id])$, we start with the identity map
$\Omega \to \Omega$ and carefully chase the above diagram to find the corresponding map $\Omega \to \mc{S^*\otimes S}$. 
Note that $\Omega\cong \mc{Q^*\otimes S}$. We have a sequence of maps
\begin{equation}\label{map:long_seq}
\begin{array}{ll}
&\Omega \xrightarrow{id} \Omega \xrightarrow{q^*}\mc{V^*\otimes S} \xrightarrow{(0,id)} \varepsilon(\mc{V^*\otimes S})\oplus (\mc{V^*\otimes S}) \to\\
\xrightarrow{T_{(0,w)}f} & \varepsilon(\mc{S^*\otimes S})\oplus (\mc{S^*\otimes S})\xrightarrow{pr_1} \mc{S^*\otimes S},
\end{array}
\end{equation}
where $q^*$ is the dual to the quotient map $\mc{S}\to\mc{V}\xrightarrow{q} \mc{Q}$, and $pr_1$ is the projection to the first factor. Let $y\in\mathbf{B}$. Recall from \eqref{map:f_B} we have
\begin{equation}
\begin{array}{ll}
f: &\mc{V^*\otimes S} \to \mc{S^* \otimes S}\\
& c^a_i(y) \mapsto c^a_i v^i_b + c^d_i (y) B^i_j(y) v^j_d \delta^a_b.
\end{array}
\end{equation}
Take the Taylor expansion of $f$ at $y = 0$ along the $w$ direction and 
discard the higher order terms to get
\begin{equation}
\begin{array}{lrl}
T_{(0,w)}f: 
&c^a_i    &\mapsto c^a_i v^i_b + \varepsilon (c^d_i \partial_w B^i_j v^j_d \delta^a_b),\\
&\varepsilon w^a_i &\mapsto \varepsilon w^a_i v^i_b.
\end{array}
\end{equation}

So, from \eqref{map:long_seq} we get \begin{equation}\label{map:omega2S*xS}
\begin{array}{lrl}
\Omega \xhookrightarrow{q^*} &\mc{V^*\otimes S} &\to \mc{S^*\times S}\\
& c^a_i &\mapsto c^d_i \partial_w B^i_j v^j_d \delta^a_b.
\end{array}
\end{equation}
This determines the class $\alpha=\alpha_B \in Hom(\Omega, \mc{S^*\otimes S})$, whose image $\delta_2(\alpha) =\delta([id])\in Ext^1(\Omega,\Omega)$.

Next we show that varying $B \in \textbf{B}$, $\alpha_B$ maps surjectively to $Ext^1(\Omega,\Omega)$. Hence  deformations in the form of \eqref{seq:deformed_omega} cover all the deformation classes.

Note that any map in $Hom(\mc{V^*\otimes S},\mc{S^*\otimes S})$ can be written as \begin{equation}
\begin{array}{rl}
\phi_{A,B}: \mc{V^*\otimes S} &\to \mc{S^*\otimes S}\\
c^a_i &\mapsto c^a_i A^i_j v^j_b + c^d_i B^i_j v^j_d \delta^a_b.
\end{array}
\end{equation}
Also, a Borel-Weil-Bott computation shows that $Hom(\Omega,\mc{V^*\otimes S}) \cong V^*\otimes V$. So any map in $Hom(\Omega,\mc{V^*\otimes S})$ is of the form $q^*_{\tilde{A}} = \tilde{A}\circ q^*$, where $\tilde{A}:\mc{V^*\otimes S}\to \mc{V^*\otimes S}$ is induced by the linear map $\tilde{A}:\mc{V^*}\to\mc{V^*}$. Hence it is easy to verify that
\begin{equation}
\phi_{A,0}\circ q^* = \phi_{I,0}\circ A \circ q^*  = \phi_{I,0}\circ q^*_A
\end{equation}
from the commutative diagrams
\begin{equation}
\xymatrix{
\Omega\ar[r]^{q^*}\ar@{=}[d] &\mc{V^*\otimes S} \ar[r]^{\phi_{A,0}} \ar[d]^{A} &\mc{S^*\otimes S} \ar@{=}[d]\\
\Omega\ar[r]^{q^*_A} &\mc{V^*\otimes S} \ar[r]^{\phi_{I,0}}  &\mc{S^*\otimes S}.
}
\end{equation}
Thus, by the first column of \eqref{diag:hom_omega}, which we recreate below, this shows that for any $\phi_{A,B}$ with $B=0$, there exists $q^*_A \in Hom(\Omega,\mc{V^*\otimes S})$, such that \begin{equation}
\delta_2(\phi_{A,0}\circ q^*) = \delta_2(\phi_{I,0}\circ q^*_A) = 0.
\end{equation}
Consider the diagram with the row and column being exact:
\begin{equation}
\xymatrix{
&Hom(\mc{V^*\otimes S},\mc{S^*\otimes S})\ar[d]^{-\circ q^*}\\
Hom(\Omega,\mc{V^*\otimes S})\ar[r]^{\phi_{I,0}\circ-} &Hom(\Omega,\mc{S^*\otimes S})\ar[d]\ar[r]^{\delta_2}&Ext^1(\Omega,\Omega)\ar[r]&0\\
&0
}
\end{equation}
The middle column ends with $0 = Ext^1(\mc{S^*\otimes S},\mc{S^*\otimes S})$\footnote{The computation boils down to 
\begin{equation}
\begin{array}{l}
\mc{S^*\otimes S^*\otimes S\otimes S} \cong \mc{O}^2\oplus (K_{1,0,...,0,-1}S^*)^4\oplus K_{2,0,...,0,-2}S^* \\
\oplus K_{1,1,0,...,0,-2}S^*\oplus K_{2,0,...,0,-1,-1}S^* \oplus K_{1,1,0,...,0,-1,-1}S^* .
\end{array}
\end{equation}
}. Similarly the $0$ on the right end is $Ext^1(\Omega, \mc{V^*\otimes S})$.

The diagram shows that the composed map $Hom(\mc{V^*\otimes S},\mc{S^*\otimes S})\to Ext^1(\Omega,\Omega)$ is surjective. So each class of $Ext^1(\Omega,\Omega)$ is represented by some $\phi_{A,B}$. Note that $\delta_2(\phi_{A,B}\circ q^*)=\delta_2((\phi_{A,0}+\phi_{0,B})\circ q^*) = \delta_2(\phi_{0,B}\circ q^*)$, so it is represented by $\phi_{0,B}$. From \eqref{map:omega2S*xS} we see $\alpha =\phi_{0,\partial_wB}\circ q^*$. But we can always pick a direction $w$ such that $\partial_w B(y)|_0$ is $B$, since the degenerate locus is a subvariety away from the origin. So we conclude that by varying $f_B$ in \eqref{seq:deformed_omega}, we have covered all Kodaira-Spencer classes.

\end{proof}

Now we turn to proving the following result:
\begin{thm}
Let $\mc{E}^*$ be the vector bundle defined by $f_B$ as in \eqref{map:f_B}.
Then $\mc{E} \cong T_X$ if and only if $B = \varepsilon I$,
where $\varepsilon$ satisfies $\varepsilon \neq-\frac{1}{k}$.
(This is the constraint for $\mc{E}^*$ to be a vector bundle, by Theorem \ref{thm:general_degen_loci}).
\end{thm}

\begin{proof}

One direction is easy: When $B =\varepsilon I$ and $\varepsilon\neq -\frac{1}{k}$,
we have a map of short exact sequences
\begin{equation}   \label{eq:trivb-diagram}
\xymatrix{
0\ar[r]&\Omega\ar[r]\ar[d] &\mc{V}^*\otimes \mc{S} \ar[r]^{f}\ar@{=}[d] &\mc{S}^*\otimes \mc{S}\ar[d]^{\cong}_{h}\ar[r]&0\\
0\ar[r]&\mc{E}^*\ar[r]&\mc{V}^*\otimes \mc{S} \ar[r]^{f_{B}} &\mc{S}^*\otimes \mc{S}
\ar[r]&0,
}
\end{equation}
where $h$ is given by $h^a_b: \sigma^a_b \mapsto \sigma^a_b + \varepsilon (\text{tr\ } \sigma) \delta^a_b$. This induces an isomorphism $\mc{E}\cong T_X$.

For the other direction, note that an isomorphism $\tilde{\sigma}$:
\begin{equation}
\xymatrix{
\mc{E}\ar[r]^{\tilde{\sigma}}\ar[d] &T_X\ar[d]\\
X\ar[r]^{\sigma} &X,
}
\end{equation}
induces a non-zero element $T_{\sigma^{-1}}\circ \tilde{\sigma} \in Hom(\mc{E}, T_X)$, where $T_{\sigma^{-1}}$ is induced from the isomorphism $\sigma^{-1}: X\to X$.
So it suffices to how that when $B^j_k\neq \varepsilon\delta^j_k$, Hom$(\mc{E},T_X) = 0$. 

We sketch the computation.

First note that Hom$(\mc{E},T_X) \cong H^0(\mc{E}^*\otimes T_X) \cong H^0(\mc{E^*\otimes Q\otimes S^*}).$ Furthermore, $\mc{E^*\otimes Q\otimes S^*}$ fits in the short exact sequence
\begin{equation}
0 \to \mc{E^*\otimes Q\otimes S^*} \to \mc{V^*\otimes S\otimes Q\otimes S^*} \to \mc{S^*\otimes S\otimes Q\otimes S^*} \to 0.
\end{equation}
To compute the cohomologies of $\mc{V^*\otimes S\otimes Q\otimes S^*} $ and $ \mc{S^*\otimes S\otimes Q\otimes S^*} $, we use two short exact sequences
\begin{equation}\label{seq:vsss}
0 \to \mc{V^*\otimes S\otimes S\otimes S^*} \to \mc{V^*\otimes S\otimes V\otimes S^*} \to \mc{V^*\otimes S\otimes Q\otimes S^*} \to 0
\end{equation}
and
\begin{equation}\label{seq:ssss}
0 \to \mc{S^*\otimes S\otimes S\otimes S^*} \to \mc{S^*\otimes S\otimes V\otimes S^*} \to \mc{S^*\otimes S\otimes Q\otimes S^*} \to 0.
\end{equation}
Using Borel-Weil-Bott, we find
\begin{equation}
\begin{array}{l}
H^0(\mc{V^*\otimes S\otimes S\otimes S^*}) = 0,\\
H^0(\mc{V^*\otimes S\otimes V\otimes S^*}) \cong V^*\otimes V,\\
H^1(\mc{V^*\otimes S\otimes V\otimes S^*}) = 0.
\end{array}
\end{equation}
The last one uses the assumption $n-k>1$.
The long exact sequence of cohomology associated to (\ref{seq:vsss}) then implies
\begin{equation}
H^0(\mc{V^*\otimes S\otimes Q\otimes S^*}) \cong H^0(\mc{V^*\otimes S\otimes V\otimes S^*}) \cong V^*\otimes V.
\end{equation}
Similarly we have
\begin{equation}
\begin{array}{l}
H^0(\mc{S^*\otimes S\otimes S\otimes S^*}) \cong H^0(\Sym^2(\mc{S^*\otimes S})) \cong \mathbb{C}^2,\\
H^0(\mc{S^*\otimes S\otimes V\otimes S^*}) \cong V^*\otimes V \oplus V^*\otimes V.
\end{array}
\end{equation}

So
\begin{equation}
H^0(\mc{E^*\otimes Q\otimes S^*}) = {\rm Ker \ } (H^0(\mc{V^*\otimes S\otimes Q\otimes S^*}) \to H^0(\mc{S^*\otimes S\otimes Q\otimes S^*}))
\end{equation}
can be computed by the diagram
{
\newcommand{\A}{H^0(\mc{V^*\otimes S \otimes V\otimes S^*})}
\newcommand{\B}{H^0(\mc{S^*\otimes S \otimes V\otimes S^*})}
\newcommand{\nA}{item}
\newcommand{\nB}{H^0(\mc{S^*\otimes S \otimes V\otimes S^*})}
\begin{equation}
\xymatrix{
\A      \ar@{->}[r]^{f_B}       &       \B
\\
&       \nB \ar[u]^{f_0}.
}
\end{equation}
Let $X^{jb}_{id}\in H^0(\mc{V^*}_{(i)}\otimes \mc{S}_{(b)} \otimes \mc{V}_{(j)}\otimes \mc{S}^*_{(d)}), Y^{jb}_{ad}\in H^0(\mc{S^*}_{(a)}\otimes \mc{S}_{(b)} \otimes \mc{V}_{(j)}\otimes \mc{S}^*_{(d)})$ and $Z^{cb}_{ad}\in H^0(\mc{S^*}_{(a)}\otimes \mc{S}_{(b)} \otimes \mc{V}_{(c)}\otimes \mc{S}^*_{(d)})$ be coordinates of the corresponding sections. The subscripts of the bundles indicate the indices used for their sections.

Let $X^{jb}_{id}=t^j_i\delta^b_d$ and $Z^{cb}_{ad} =u_1\delta^c_a\delta^b_d + u_2 \delta^b_a\delta^c_d$. With the concrete expressions of $f_B$ and $f_0$ (which is $f_B$ with $B=0$) as in (\ref{map:f_B}), it is then straightforward to show that $$H^0(\mc{E^*\otimes Q\otimes S^*})\neq 0 {\rm\ iff\ } B^j_k \neq \varepsilon \delta^j_k , {\rm \ where\ }\varepsilon = \frac{u_2}{u_1}.$$
}
\end{proof}
A similar argument shows that $Hom(T_X, \mc{E})$ is always non-trivial. Together with the fact that the tangent bundle of the Grassmannian $G(k,n)$ is stable, we have
\begin{thm}
When $B\neq \varepsilon I$, the corresponding deformed tangent bundle $\mc{E}$ is not Gieseker semistable.
\end{thm}
\begin{proof}
This is a direct corollary of \cite{huybrechts1997geometry}, Proposition 1.2.7.
\end{proof}

\section{The polymology as a quotient}\label{sec:quotient_commutativity}
From the polymology of the tangent bundle and  semi-continuity we know that
for generic deformations $\mc{E}$ of the tangent bundle, $H^q(\wedge^p\mc{E}^*) = 0$ for $p\neq q$ and  we only need to focus on $H^r(\wedge^r\mc{E}^*)$. This could go wrong on some subvariety of the $B$-parameter space. We call this subvariety $\mc{B}_{\text{jump}}$ and study it in Section \ref{sec:non-generic}.

Since
\begin{equation}
 0 \to \mc{E}^* \to \mc{V^*\otimes S} \to \mc{S^*\otimes S}\to 0.
 \end{equation}
Taking the Koszul resolution, we have the long exact sequence
\begin{equation}\label{seq:Koszul}
0 \to \wedge^r \mc{E}^* \to \wedge^r (\mc{V^*\otimes S}) \to \wedge^{r-1} (\mc{V^*\otimes S}) \otimes ... \to \Sym^r(\mc{S^*\otimes S}) \to 0.
\end{equation}

We then try to describe the cohomology of $\wedge^r\mc{E}^*$ by that of $\mc{S^*\otimes S}$.

We use a familiar strategy: break \eqref{seq:Koszul} into short exact sequences and in the end it boils down to understand the kernel of $H^0(\Sym^r(\mc{S^*\otimes S})) \to H^r(\wedge^r\mc{E}^*)$.

Moreover, this leads to a description of the polymology as a quotient ring:

\begin{thm}\label{quotient ring thm}
The polymology ring is a quotient of $H^0(\Sym^*(\mc{S}^*\otimes\mc{S}))$.
\end{thm}
\begin{proof}
We need to show that the following diagram is commutative:

{
\newcommand{\A}{H^0(\Sym^s(\mc{S^*\otimes S})) \times H^0(\Sym^t(\mc{S^*\otimes S}))}
\newcommand{\B}{H^0(\Sym^s(\mc{S^*\otimes S})) \otimes H^0(\Sym^t(\mc{S^*\otimes S}))}
\newcommand{\nA}{H^s(\wedge^s \mc{E}^*) \times  H^t(\wedge^t \mc{E}^*)}
\newcommand{\nB}{H^s(\wedge^s \mc{E}^*) \otimes  H^t(\wedge^t \mc{E}^*)}
\newcommand{\C}{H^0(\Sym^s(\mc{S^*\otimes S})\otimes\Sym^t(\mc{S^*\otimes S}))}
\newcommand{\nC}{H^{s+t}(\wedge^s\mc{E}^*\otimes\wedge^t\mc{E}^*)}
\newcommand{\D}{H^0(\Sym^{s+t}(\mc{S^*\otimes S}))}
\newcommand{\nD}{H^{s+t} (\wedge^{s+t}\mc{E}^*)}
\begin{equation}\label{diag:quotient_ring_structure}
\xymatrix{
\A      \ar[r]  \ar[d]      &       \D  \ar[d]
\\
\nA     \ar[r]              &       \nD.
}
\end{equation}
}

We show it by proving the commutativity of two diagrams,
{
\newcommand{\A}{H^0(\Sym^s(\mc{S^*\otimes S})) \times H^0(\Sym^t(\mc{S^*\otimes S}))}
\newcommand{\B}{H^0(\Sym^s(\mc{S^*\otimes S})) \otimes H^0(\Sym^t(\mc{S^*\otimes S}))}
\newcommand{\nA}{H^s(\wedge^s \mc{E}^*) \times  H^t(\wedge^t \mc{E}^*)}
\newcommand{\nB}{H^s(\wedge^s \mc{E}^*) \otimes  H^t(\wedge^t \mc{E}^*)}
\newcommand{\C}{H^0(\Sym^s(\mc{S^*\otimes S})\otimes\Sym^t(\mc{S^*\otimes S}))}
\newcommand{\nC}{H^{s+t}(\wedge^s\mc{E}^*\otimes\wedge^t\mc{E}^*)}
\newcommand{\D}{H^0(\Sym^{s+t}(\mc{S^*\otimes S}))}
\newcommand{\nD}{H^{s+t} (\wedge^{s+t}\mc{E}^*)}
\begin{equation}\label{diag:delta_functor}
\xymatrix{
\A      \ar@{->}[r] \ar[d]      &       \C  \ar[d]
\\
\nA     \ar[r]              &       \nC
}
\end{equation}
and
\begin{equation}\label{diag:tensor_to_exterior}
\xymatrix{
\C      \ar[r]  \ar[d]^{\delta^{s+t}_\otimes}       &       \D  \ar[d]^{\delta^{s+t}}
\\
\nC     \ar[r]              &       \nD.
}
\end{equation}

\textit{1. The commutativity of diagram (\ref{diag:delta_functor}):}

This is induced from the diagram on the Cech cocycle level,

\newcommand{\AZ}{Z^0(\Sym^s(\mc{S^*\otimes S})) \times Z^0(\Sym^t(\mc{S^*\otimes S}))}
\newcommand{\nAZ}{Z^s(\wedge^s \mc{E}^*) \times  Z^t(\wedge^t \mc{E}^*)}
\newcommand{\CZ}{Z^0(\Sym^s(\mc{S^*\otimes S})\otimes\Sym^t(\mc{S^*\otimes S}))}
\newcommand{\nCZ}{Z^{s+t}(\wedge^s\mc{E}^*\otimes\wedge^t\mc{E}^*)}

\begin{equation}
\xymatrix{
\AZ     \ar@{->}[r] \ar[d]      &       \CZ \ar[d]
\\
\nAZ        \ar[r]              &       \nCZ
,}
\end{equation}
whose commutativity can be directly verified via the following two squares:

\newcommand{\nnAZ}{Z^s(\wedge^s \mc{E}^*) \times  Z^0(\Sym^t(\mc{S^*\otimes S}))}
\newcommand{\nnCZ}{Z^s(\wedge^s\mc{E}^*\otimes\Sym^t(\mc{S^*\otimes S}))}

\begin{equation}
\xymatrix{
\AZ     \ar@{->}[r] \ar[d]^{\delta^s}       &       \CZ \ar[d]^{\delta^s}
\\
\nnAZ       \ar[r] \ar[d]^{\delta^t}                &       \nnCZ   \ar[d]^{\delta^t}
\\
\nAZ        \ar[r]              &       \nCZ
.}
\end{equation}

\textit{2. The commutativity of diagram (\ref{diag:tensor_to_exterior}):}

First observe that $\delta^{s+t}$ is induced by the Koszul sequence (\ref{seq:Koszul}) with $r=s+t$. To see this, we break this long exact sequence into short exact sequences
\begin{equation}
\begin{array}{l}
0 \to S_r \to Z_r \to S_{r-1} \to 0, \\
0 \to S_{r-1} \to Z_{r-1} \to S_{r-2} \to 0,\\
...,\\
0 \to S_1 \to Z_1 \to S_0 \to 0,
\end{array}
\end{equation}
where $Z_j = \wedge^j (\mc{V^*\otimes S})\otimes \Sym^{r-j} (\mc{S^*\otimes S})$, $S_j = {\rm Ker\ } (Z_j \to Z_{j-1})$, and $S_0 =\Sym^r(\mc{S^*\otimes S})$. They induce connecting maps on cohomology $\delta: H^{j}(S_j) \to H^{j+1}(S_{j+1})$, $j=0,...,r-1$ and $\delta^{s+t}$ is the composition of them.

The maps
$\delta^{s+t}_\otimes$ is induced by a similar long exact sequence. To describe it, we need to rephrase the Koszul resolution in the language of the Schur complexes. Consider the complex
\begin{equation}
\mathbb{E}: \mc{S\otimes S^*} \xrightarrow{f} \mc{V\otimes S^*} ,
\end{equation}
which defines $\mc{E}$ as Coker $f$. The Schur complex $L_r\mathbb{E}$ is defined as
\begin{equation}
L_r \mathbb{E}: \Sym^r (\mc{S\otimes S^*}) \to \Sym^{r-1} (\mc{S\otimes S^*}) \otimes (\mc{V\otimes S^*}) \to ... \to \wedge^r (\mc{V\otimes S^*}).
\end{equation}

More general Schur complexes $L_\lambda\mathbb{E}$ are indexed by Young
diagrams $\lambda$. See \cite{weyman2003cohomology}, Section 2.4.
The tensor of two Schur complexes satisfies the Littlewood-Richardson rule
(see \cite{weyman2003cohomology} Remark (2.4.8 - b), also note that the
Schur functors commute with the differentials of the Schur complex
\cite{weyman2003cohomology} 2.4.10): $L_\lambda\mathbb{E} \otimes L_\mu\mathbb{E} = \oplus c^\gamma_{\lambda\mu} L_\gamma \mathbb{E}$.
In particular, we have
$$L_s\mathbb{E}\otimes L_t \mathbb{E} = L_{s+t}\mathbb{E} \oplus
\mbox{(other terms)}.$$
This induces a map of complexes
\begin{equation}
 u: L_{s+t}\mathbb{E} \hookrightarrow L_s\mathbb{E}\otimes L_t \mathbb{E}.
 \end{equation}

Notice that $L_s\mathbb{E}\otimes L_t \mathbb{E}$ gives rise to a long exact
sequence
\begin{equation}
0 \to L_s\mathbb{E}\otimes L_t \mathbb{E} \to \wedge^{s}\mc{E}\otimes \wedge^t\mc{E} \to 0.
\end{equation}
Dualizing it, we get the long exact sequence
\begin{equation}\label{seq:tensor}
0 \to \wedge^{s}\mc{E}^*\otimes \wedge^t\mc{E}^* \to (L_s\mathbb{E}\otimes L_t \mathbb{E})^\vee \to 0,
\end{equation}
i.e.
\begin{equation}\begin{array}{ll}
0 &\to \wedge^{s}\mc{E}^*\otimes \wedge^t\mc{E}^* \to \wedge^{s}(\mc{V^*\otimes S})\otimes \wedge^t(\mc{V^*\otimes S})\to...\\
&\to \Sym^s(\mc{S^*\otimes S})\otimes \Sym^t(\mc{S^*\otimes S})\to 0.
\end{array}
\end{equation}
This sequence induces the map $\delta^{r+t}_\otimes$ on cohomology.

Similarly, we have
\begin{equation}\label{seq:koszul_s+t}
0 \to \wedge^{s+t}\mc{E}^* \to L_{s+t}\mathbb{E}^\vee \to 0,
\end{equation}
which is exactly the Koszul resolution (\ref{seq:Koszul}) with $r=s+t$.

Then it is easy to see that the dual of the map $u$ extends to a map of complexes (\ref{seq:tensor}) to (\ref{seq:koszul_s+t}). Since the connecting morphisms $\delta^{s+r}_\otimes$ and $\delta^{s+t}$ are functorial, this proves the commutativity of (\ref{diag:tensor_to_exterior}).

Combining the two diagrams, we get the desired commutative diagram (\ref{diag:quotient_ring_structure}).
}

\end{proof}

\begin{rem}\label{rmk:mult}
This theorem enables us to compute the multiplicative structure using that of
$H^0(\Sym^*(\mc{S^*\otimes S}))$.  By virtue of the fact that the
Schur functors obey
\begin{displaymath}
K_{\lambda} \mc{S} \otimes K_{\mu} \mc{S} =
\sum c^{\nu}_{\lambda \mu} K_{\nu} \mc{S},
\end{displaymath}
where $c^\nu_{\lambda\mu}$ is the Littlewood-Richardson coefficient,
we know that for
$$\kappa_\lambda, \kappa_\mu \in  H^0(\Sym^*(\mc{S^*\otimes S})),$$
we have
$$\kappa_\lambda\cdot\kappa_\mu = \sum c^\nu_{\lambda\mu}\kappa_\nu.$$
\end{rem}

\begin{rem}
In particular, the polymology ring is isomorphic to the ring of symmetric
polynomials in $k$ indeterminates.  The section
$$\kappa_{\lambda} \: \in \: H^0( {\rm Sym}^r(\mc{S^* \otimes S}))$$
corresponds to a Schur polynomial in those indeterminates associated to
the Young diagram $\lambda$.
\end{rem}

\section{The cohomology of $\wedge^r\mc{E}^*$}\label{sec:cohomology_computation}

We then want to describe $H^r(\wedge^r\mc{E}^*)$ via $\Delta = \Delta_B : H^0(\Sym^r(\mc{S^*\otimes S})) \to H^r(\wedge^r\mc{E}^*)$. Denote the kernel of $\Delta: H^0(\Sym^r(\mc{S}^*\otimes \mc{S})) \to H^r(\wedge^r\mc{E}^*)$ as $\mathbb{K}_r$.

\subsection{$B$-dependence}
\begin{thm}\label{similarity}
The kernel $\mathbb{K}_r$ of $\Delta: H^0(\Sym^r(\mc{S}^*\otimes \mc{S})) \to H^r(\wedge^r\mc{E}^*)$ only depends on the equivalence class of $B$ modulo similarity transformations $B\mapsto g B g^{-1},  g\in GL(V)$.
\end{thm}
The proof is clear from the construction and so is omitted.

Now we consider the image of $\sigma\in H^0(\Sym^2(\mc{S}^*\otimes\mc{S}))$ under $\Delta_B$.
To track the $B_{ij}$-dependence, we make the following definition.
\begin{defn}
Let each $B_{ij}$ be a degree one variable and denote the
\textit{total B degree} of each cocycle $\omega$ as $\deg \omega$.
\end{defn}
For any $n\times n$ matrix $B$, consider the characteristic polynomial
(with the sign changed) $\det(\lambda I + B)$.  Denote the coefficient of
$\lambda^{n-i}$ as $I_i(B)=I_i$, so that $I_0=1, I_1 = tr (B),
I_2 =\frac{1}{2}(tr(B)^2 - tr(B^2))),$ and so forth. We have $\deg I_i = i$.

\begin{thm}\label{B_ij-degree}
Every $\sigma\in \mathbb{K}_r$
is determined by some $\gamma \in {\rm Ker}(H^{j-1}(Z_j^{(r)}) \to H^{j-1}(Z_{j-1}^{(r)}))$, where $Z_{j}^{(r)} = \wedge^j (\mc{V}^*\otimes\mc{S}) \otimes \Sym^{r-j}(\mc{S}^*\otimes\mc{S})$.\footnote{For any $\sigma$, there always exists such $\gamma$.}
If $\gamma$ is $B$-independent, then $\sigma$ can be represented by a cocycle $\gamma_0$ such that $\deg \gamma_0 \leq r$.
\end{thm}
\begin{proof}

The first half of the theorem is easily verified by applying Lemma \ref{lemma:tech-app} to the
sequence \eqref{the_true_LEG}.

Under the assumption we have $\deg \lambda =0$ as it is $B$-independent, and $\deg q_i = 1$ by linearity.
The key observation here is $\deg d^{-1} =0$, i.e. $\deg \gamma^0 = \deg \gamma^1$. Assume this is not true,
then  $\deg\gamma^0 > \deg \gamma^1$. Take the sum of the terms of $\gamma^0$ whose degrees are larger than $\deg
 \gamma^1$, call it $\gamma^0_+$, then $d(\gamma^0_+) $ has to
 be 0 $\in C^1((\mc{V}^*\otimes\mc{S})\otimes (\mc{S}^*\otimes\mc{S}))$ since there is no term of the corresponding degree there.
 Hence one can simply drop $\gamma^0_+$ when choosing $\gamma^0$.
\end{proof}

\begin{cor}
When $\gamma$ is $B$-independent, the kernel $\mathbb{K}_r$
is generated by elements with
coefficients that are polynomials of $I_1,...,I_r$ whose total $B$ degrees are less than or equal to $r$.
\end{cor}

\subsection{Generalities of the $r=n-k+1$ case}

By Remark \ref{rem1} and Theorem \ref{vanishing_r_leq_n-k}, the first case we expect a nontrivial kernel of $\Delta: H^0(S_0)\to H^r(S_r)$ is the case when $r=n-k+1$. Here we used the notion $S_i$ as in the short exact sequences \begin{equation}
0\to S_j \to Z_j \to S_{j-1}\to 0,
\end{equation}
$j=1,..., r$, which are generated from the long exact sequence (\ref{seq:Koszul}). In particular, $S_0 =\Sym^r(\mc{S}^*\otimes\mc{S})$, $S_r=\wedge^r\mc{E}^*$,
$Z_j = \wedge^j(\mc{V}^*\otimes \mc{S})
\otimes \Sym^{r-j}(\mc{S^*\otimes S})$.

We find that
\begin{thm}\label{glV_invariance_r=n-k+1}
When $r=n-k+1$, $\mathbb{K}_r$, the kernel of $\Delta: H^0(S_0)\to H^r(S_r)$ is generated by the image of a $GL(V)$-invariant element in $H^{r-1}(\wedge^{r}(\mc{V^*\otimes S}))$,
 for any $B-$deformed $\mc{E}^*$.
\end{thm}
\begin{proof}
Consider the morphism of complexes
\begin{equation}\label{map_of_ses}
\xymatrix{
0\ar[r]&\mc{E}^*\ar[r]\ar[d]^{g}&\mc{V}^*\otimes \mc{S} \ar[r]^{f_B}\ar@{=}[d]^{id} &\mc{S}^*\otimes \mc{S}\ar[d]^{g_0}\ar[r]&0\\
0\ar[r]&\mc{E}^*_0\ar[r]&\mc{V}^*\otimes \mc{S} \ar[r]^{f_{0}\ \ \ } &End_0\ \mc{S}
\ar[r]&0
.}
\end{equation}
Take the induced long exact sequences, we have
\begin{equation}\label{wedgeE2wedgeE0}
\xymatrix{
0\ar[r]&\wedge^r\mc{E}^*\ar[r]\ar[d]^{g}&...\ar[r]&Z_j\ar[d]\ar[r]&...\ar[r] &\Sym^r(\mc{S}^*\otimes \mc{S})\ar[d]\ar[r]&0\\
0\ar[r]&\wedge^r\mc{E}^*_0\ar[r]&...\ar[r]&Z_{0,j}\ar[r]&...\ar[r] &\Sym^r (End_0\ \mc{S})
\ar[r]&0
,}
\end{equation}
where $Z_{0,j} =  \wedge^j(\mc{V}^*\otimes \mc{S})
\otimes \Sym^{r-j}(End_0\ \mc{S})$.

We claim that the vertical arrows induces isomorphisms on cohomologies, for $j=1,...,r$. First, for $j=r$ this is identity. Then, for $j=1,...,r-1$, note that
$\mc{S^*\otimes S} \cong End_0 \mc{S}\oplus \mc{O}$. Hence
\begin{equation}\label{iso_r}
\begin{array}{ll}
\Sym^{r-j}(\mc{S^*\otimes S}) & \cong \Sym^{r-j}(End_0 \mc{S}) \oplus ...\oplus \Sym^2(End_0\mc{S}) \oplus End_0 \mc{S} \oplus \mc{O},\\
&\cong \Sym^{r-j}(End_0 \mc{S}) \oplus \Sym^{r-j-1}(\mc{S^*\otimes S}).
\end{array}
\end{equation}
Since $H^\bullet( \wedge^j(\mc{V}^*\otimes \mc{S})\otimes\Sym^{r-j-1}(\mc{S^*\otimes S}))=0$ by Theorem \ref{vanishing_r_leq_n-k}
\footnote{Or rather its variant, that with the assumption of the theorem, for each $\mu$ such that $0\subseteq\mu\subsetneq \lambda$, we have
$H^\bullet(K_{\mu}\mc{S}^*\otimes K_\lambda \mc{S}) =0$. This is stated in the proof of the theorem.}
, the claim is proved by tensoring (\ref{iso_r}) with $ \wedge^j(\mc{V}^*\otimes \mc{S}) $.

This implies that the kernels of $H^{j-1}(Z_{j}) \to H^{j-1}(Z_{j-1})$ for $\wedge^r\mc{E}^*$ are all isomorphic to the corresponding ones for $\Omega^r$, via the squares
\begin{equation}\label{n-k+1_Z_r}
 \xymatrix{
H^{j-1}(Z_{j})\ar[r]\ar[d]^{\cong} & H^{j-1}(Z_{j-1})\ar[d]^{\cong}\\
H^{j-1}(Z_{0,j})\ar[r] & H^{j-1}(Z_{0,j-1})
 .}
 \end{equation}

We then proceed to check the $\Omega^r$ case.

For the cotangent bundle, (\ref{seq:Koszul}) is a term-by-term direct sum of long exact sequences tensoring with $K_\lambda \mc{S}$, as shown in Remark \ref{rem1}. Only the $\lambda$'s with $\lambda_1 > n-k$ will contribute to
${\rm Ker} \, \Delta$, by Theorem \ref{vanishing_r_leq_n-k}. When $r=n-k+1$, this means we only need to consider the case $\lambda =(r)$, i.e. the long exact sequence reduces to \begin{equation}\begin{array}{ll}
0 &\to 0 \to \wedge^r \mc{V}^* \otimes\Sym^r\mc{S}\to ... \to \wedge^{j} \mc{V}^*\otimes \Sym^{r-j}\mc{S}^*\otimes\Sym^r\mc{S}\\
& \to ... \to \Sym^r\mc{S}^*\otimes\Sym^r\mc{S}\to 0
\end{array}
\end{equation}
for the purpose of computing ${\rm Ker} \, \Delta$.

By Borel-Weil-Bott, $H^i(Z_j)=H^i(\mc{V}^*\otimes \Sym^{r-j}\mc{S}^*\otimes\Sym^r\mc{S}) =0$ for $i< n-k$, $j=1,...,n-k$\footnote{
$\Sym^{r-j}\mc{S}^*\otimes\Sym^r\mc{S} = \Sym^{r-j}\mc{S}^*\otimes\otimes K_{(r^{k-1})}\mc{S}^*\otimes (\wedge^k\mc{S})^r$ can be completely determined by Pieri's formula. We just need the fact that,
when $j=1,...,n-k$, for any component $K_\lambda \mc{S}^*$ of $\Sym^{r-j}\mc{S}^*\otimes\Sym^r\mc{S}$, we have $|\lambda|=\sum_{i=1}^k\lambda_i=-j$. So $\lambda_k <0$. So it takes at least $n-k$ steps to mutate
${(\lambda_1,...,\lambda_k,0^{n-k})}$ to a decreasing sequence.
}.

So the only contribution to ${\rm Ker} \, \Delta$ comes from the kernel of $$H^{r-1}(Z_r)\xrightarrow{\bar{f}_B} H^{r-1}(Z_{r-1}),$$ which is the $GL(V)$ invariant part of $H^{r-1}(Z_r) = \wedge^r V^*\otimes \wedge^r V$ (which is $K_0 V^* = \mathbb{C}$).
Note that the identity map in the middle column of \eqref{map_of_ses} induces an identity map  \begin{equation}
H^{r-1}(Z_r)\to H^{r-1}(Z_{0,r}).
\end{equation}
Hence we have
\begin{equation}
 \xymatrix{
{\rm Ker} f_B \ar[r]\ar[d]&H^{r-1}(Z_{r})\ar[r]^{\bar{f}_B}\ar@{=}[d]^{}& H^{r-1}(Z_{r-1})\ar[d]^{\cong}\\
{\rm Ker} f_0 \ar[r]&H^{r-1}(Z_{0,r})\ar[r]^{\bar{f}_0} & H^{r-1}(Z_{0,r-1})
 .}
 \end{equation}

So we have ${\rm Ker}\, \bar{f}_B = {\rm Ker} \, \bar{f}_0$, for any $B$.

\end{proof}

Let $V=V_1\oplus L$ be an $n$ dimensional vector space. Consider the inclusion of Grassmannians $X=G(k-1,V_1)\hookrightarrow Y = G(k,V)$, with $[S_1]\mapsto [S_1\oplus L]$. Note that in this case we have $\mc{V}|_X = \mc{V}_1\oplus \mc{L}$, $\mc{S}|_X=\mc{S}_1\oplus \mc{L}$, and similarly for their duals. We extend the $GL(V_1)$ action to $V$ by making $L$ a trivial $GL(V_1)$ module. This will be implicitly used when considering the $GL(V_1)$ invariant parts of cohomologies.

\begin{lem}\label{CD_gk-1n-1}
Let $B=\begin{pmatrix}
B_1 & 0\\ 0 & 0
\end{pmatrix}$.
Then there is a commutative diagram:
{
\newcommand{\A}{\mc{V^*\otimes S}|_X}
\newcommand{\B}{\mc{S^*\otimes S}|_X}
\newcommand{\nA}{\mc{V}_1^*\mc{\otimes S}_1}
\newcommand{\nB}{\mc{S}_1^*\otimes \mc{S}_1}
\begin{equation}
\xymatrix{
\A      \ar[r]^{f_B}    \ar[d]^{\pi}            &       \B  \ar[d]^{\pi}
\\
\nA     \ar[r]^{f_{B_1}}                &       \nB
,}
\end{equation}
}
given by the natural projections as vertical maps.
\end{lem}
\begin{proof}
We take the standard basis for $V=V_1\oplus L$, with $V_1=\langle e_1,...,e_{n-1}\rangle$ and $L=\langle e_n\rangle$. On $X$, each $S$ is generated by $v_1,...,v_k$ with $v_k = (0,...,0,1)^T$, and $v^n_b = 0$ for $b\leq k-1$.

As before, we know the map $f_B$ and $f_{B_1}$ explicitly. \begin{equation}
f_B: c^a_i \mapsto c^a_i v^i_b +c^d_i B^i_j v^j_d \delta^a_b,
\end{equation}
and similarly for $f_{B_1}$ with $a,i$ indices runs to $k-1,n-1$ instead of $k,n$. Hence
\begin{equation}
\begin{array}{ll}
\pi \circ f_B - f_{B_1}\circ\pi
&=c^a_n v^n_b + c^k_i B^i_j v^j_k \delta^a_b + c^d_n B^n_j v^j_d \delta^a_b + c^d_i B^i_n v^n_d \delta^a_b
\\
&= c^k_i B^i_n v^n_k \delta^a_b \text{\ \ (since\ } v^j_k = \delta^j_n \text{)}
\\
&=0.
\end{array}
\end{equation}
\end{proof}
Together with the commutative diagram
{
\newcommand{\A}{\mc{V^*\otimes S}}
\newcommand{\B}{\mc{S^*\otimes S}}
\newcommand{\nA}{\mc{V^*\otimes S}|_X}
\newcommand{\nB}{\mc{S^*\otimes S}|_X}
\begin{equation*}
\xymatrix{
\A      \ar[r]  \ar[d]      &       \B  \ar[d]
\\
\nA     \ar[r]              &       \nB
}
\end{equation*}
}
from natural restrictions,
we get
{
\newcommand{\A}{\mc{V^*\otimes S}}
\newcommand{\B}{\mc{S^*\otimes S}}
\newcommand{\nA}{\mc{V^*\otimes S}|_X}
\newcommand{\nB}{\mc{S^*\otimes S}|_X}
\newcommand{\nC}{\mc{V}_1^*\mc{\otimes S}_1}
\newcommand{\nD}{\mc{S}_1^*\otimes \mc{S}_1}

\begin{equation}
\xymatrix{
\A      \ar[r]^{f_B}    \ar[d]      &       \B  \ar[d]^{q}
\\
\nA     \ar[r]^{f_B}    \ar[d]      &       \nB \ar[d]
\\
\nC     \ar[r]^{f_{B_1}}                &       \nD
.}
\end{equation}
}
Note that each horizontal line is surjective for suitable $B$ or $B_1$, with a vector bundle as its kernel. In particular, the second line is so because both the first line and $q$ are surjective. This can also be seen from restricting the first line to $X$ as vector bundles directly.
So this induces maps of Koszul complexes similar to (\ref{wedgeE2wedgeE0}), and further the following commutative diagram:
{
\newcommand{\A}{H^{r-1}(\wedge^r(\mc{V^*\otimes S}))_0}
\newcommand{\B}{H^0(\Sym^r(\mc{S^*\otimes S}))}
\newcommand{\nA}{H^{r-1}(\wedge^r(\mc{V^*\otimes S})|_X)_0}
\newcommand{\nB}{H^0(\Sym^r(\mc{S^*\otimes S})|_X)}
\newcommand{\nC}{H^{r-1}(\wedge^r(\mc{V}_1^*\mc{\otimes S}_1))_0}
\newcommand{\nD}{H^0(\Sym^r(\mc{S}_1^*\otimes \mc{S}_1))}

\begin{equation}\label{CDq0}
\xymatrix{
\A      \ar[r]  \ar[d]^{q_{r-1}}        &       \B  \ar[d]^{q_0}
\\
\nA         \ar[d]^{\pi_{r-1}}      &       \nB \ar[d]^{\pi_0}
\\
\nC     \ar[r]              &       \nD
,}
\end{equation}
where the $0$ in the first line indicates $GL(V)$ invariance and the $0$'s in the second and third line indicate $GL(V_1)$ invariance.

The first line and the third line are clear from Theorem \ref{glV_invariance_r=n-k+1}, with the induced map $\A \to \nC$. Observe that the map factors through a subspace of $H^{r-1}(\wedge^r(\mc{V^*\otimes S})|_X),$ which is the preimage of $\nC $. So it has to be $\nA$.
}

Recall that $\kappa_\lambda$ is the canonical generator of $H^0(K_\lambda \mc{S}^*\otimes K_\lambda \mc{S})$. We will use $\kappa_{\lambda,Y}$ to indicate the base manifold $Y$.
\begin{lem}
The map $\pi_0\circ q_0$ maps $\kappa_{\lambda,Y}$ to $\kappa_{\lambda,X}$.
\end{lem}
\proof{The natural decomposition $\Sym^r(\mc{S^*\otimes S}) \cong \sum_{\lambda}K_\lambda \mc{S}^*\otimes K_\lambda \mc{S}$ implies that it suffices to consider
$$K_\lambda \mc{S}^*\otimes K_\lambda \mc{S} \xrightarrow{q_0} K_\lambda \mc{S}^*\otimes K_\lambda \mc{S}|_X \xrightarrow{\pi_0} K_\lambda \mc{S}_1^*\otimes K_\lambda \mc{S}_1.$$

We give the explicit expression of $\kappa_\lambda$ using the normalized Young symmetrizer $$c_\lambda = n_\lambda \displaystyle\sum_{g\in R(T), h\in C(T)}\text{sgn}(h) e_{gh}$$ for a Young Tableau $T$ of shape $\lambda$ (we actually do not impose any increasing row / column condition on $T$, so $T$ is just a filling of $\lambda$ with $1,...,r$ ). Recall that one way to define $K_\lambda S$ over complex numbers is $K_\lambda S = \text{Im\ } c_\lambda (S^{\otimes r})$
(see Section 6.1 of Fulton-Harris \cite{MR1153249}),
where $n_\lambda$ is a number.

It is straightforward to verify that
\begin{equation}
\kappa_\lambda =\frac{n_\lambda}{r!}\sum_{g,h, \tau}\text{sgn}(h) \delta^{a_{\tau(1)}}_{b_{\tau\rho(1)}} \cdots \delta^{a_{\tau(r)}}_{b_{\tau\rho(r)}} v_{a_1}\otimes \cdots \otimes v^{b_r},
\end{equation}
where the summation is for $g\in R(T),h\in C(T), \tau\in S_r$ with $\rho=gh$.

Then we observe that $q_0(\kappa_\lambda) = \kappa_\lambda$, and the effect of the projection $\pi_0$ is just changing the summation ranges of $a_i, b_i$ from $\{1,...,k\}$ to $\{1,..., k-1\}$. This proves $\pi_0\circ q_0(\kappa_{\lambda,Y}) = \kappa_{\lambda,X}$.
}

To understand $\pi_{r-1}\circ q_{r-1}$, we consider the following commutative diagram:
{
\newcommand{\A}{\mc{S}}
\newcommand{\B}{\mc{V}}
\newcommand{\nA}{\mc{S}|_X}
\newcommand{\nB}{\mc{V}|_X}
\newcommand{\nC}{\mc{S}_1 }
\newcommand{\nD}{\mc{V}_1}

\begin{equation}
\xymatrix{
\A      \ar[r]  \ar[d]^{}       &       \B  \ar[d]
\\
\nA     \ar[r]  \ar[d]      &       \ \ \nB \ar[d]
\\
\nC     \ar[r]              &       \nD
.}
\end{equation}
}

The first line induces
\begin{equation}
0\to \Sym^r\mc{S} \to \Sym^{r-1}\mc{S}\otimes \mc{V}\to...\to \wedge^r\mc{V} \to \wedge^r\mc{Q}=0
\end{equation}
and hence an isomorphism $H^{r-1}(\Sym^r\mc{S})\cong H^0(\wedge^r\mc{V})$ (from the vanishing of the cohomologies of the terms in between), which in turn indicates
$H^{r-1}(\Sym^r\mc{S}\otimes \wedge^r\mc{V^*})\cong H^0(\wedge^r \mc{V}\otimes \wedge^r \mc{V^*})$.

We then have

{
\newcommand{\B}{H^{r-1}(\Sym^r\mc{S}\otimes \wedge^r\mc{V^*})}
\newcommand{\A}{H^0(\wedge^r \mc{V}\otimes \wedge^r \mc{V^*})}
\newcommand{\nB}{H^{r-1}(\Sym^r\mc{S}\otimes \wedge^r\mc{V^*}|_X)}
\newcommand{\nA}{H^0(\wedge^r \mc{V}\otimes \wedge^r \mc{V^*}|_X)}
\newcommand{\nD}{H^{r-1}(\Sym^r\mc{S}_1\otimes \wedge^r\mc{V}_1^*)}
\newcommand{\nC}{H^0(\wedge^r \mc{V}_1\otimes \wedge^r \mc{V}_1^*)}

\begin{equation}
\xymatrix{
\A      \ar[r]^{\cong}  \ar[d]^{ }      &       \B  \ar[d]
\\
\nA             \ar[d]      &       \ \ \nB \ar[d]
\\
\nC     \ar[r]^{\cong}              &       \nD
.}
\end{equation}
Since $\B$ is the only non-vanishing part of $H^{r-1}(\wedge^r(\mc{V^*\otimes S}))$, we actually have
}

{

\newcommand{\A}{H^0(\wedge^r \mc{V}\otimes \wedge^r \mc{V^*})_0}
\newcommand{\B}{H^{r-1}(\wedge^r(\mc{V^*\otimes S}))_0}
\newcommand{\nA}{H^0(\wedge^r \mc{V}\otimes \wedge^r \mc{V^*}|_X)_0}
\newcommand{\nB}{H^{r-1}(\wedge^r(\mc{V^*\otimes S})|_X)_0}
\newcommand{\nC}{H^0(\wedge^r \mc{V}_1\otimes \wedge^r \mc{V}_1^*)_0}
\newcommand{\nD}{H^{r-1}(\wedge^r(\mc{V}_1^*\mc{\otimes S}_1))_0}

\begin{equation}
\xymatrix{
\A      \ar[r]^{\cong}  \ar[d]^{q_{r-1}^\prime}     &       \B  \ar[d]^{q_{r-1}}
\\
\nA             \ar[d]^{\pi_{r-1}^\prime}       &       \ \ \nB \ar[d]^{\pi_{r-1}}
\\
\nC     \ar[r]^{\cong}              &       \nD
.}
\end{equation}
}
Note that we take the $GL(V)$ and $GL(V_1)$ invariant parts as before.

\begin{lem}\label{lem:iso}
$\pi_{r-1}\circ q_{r-1}$ is an isomorphism.
\end{lem}
\begin{proof}
It suffices to prove that $\pi_{r-1}^\prime\circ q_{r-1}^\prime$ is an isomorphism. This can be done by direct computation. Note that $q_{r-1}^\prime$ is induced from the identity map $$H^0(\wedge^r \mc{V}\otimes \wedge^r \mc{V^*}) \to H^0(\wedge^r \mc{V}\otimes \wedge^r \mc{V^*}|_X),$$ which is $$\wedge^r V\otimes \wedge^r V^* \to \wedge^r V\otimes \wedge^r V^*, $$ and $\pi_{r-1}^\prime$ is induced from the projection to $(\wedge^r V_1\otimes \wedge^r V^*_1)$.

It is then straight forward to observe that $\pi_{r-1}^\prime\circ q_{r-1}^\prime$ maps
$$\displaystyle\sum_{a_i,b_i =1}^{n}\sum_{\rho \in S_r} (-1)^\rho \delta^{a_1}_{b_{\rho(1)}} \cdots \delta^{a_r}_{b_{\rho(r)}}e_{a_1}\otimes \cdots \otimes e_{a_r}\otimes e^{b_1}\otimes \cdots \otimes e^{b_r},$$
 the generator of the one dimensional space $H^0(\wedge^r \mc{V}\otimes \wedge^r \mc{V^*})_0$, to
$$\displaystyle\sum_{a_i,b_i =1}^{n-1}\sum_{\rho \in S_r} (-1)^\rho \delta^{a_1}_{b_{\rho(1)}} \cdots \delta^{a_r}_{b_{\rho(r)}}e_{a_1}\otimes \cdots\otimes e_{a_r}\otimes e^{b_1}\otimes \cdots\otimes e^{b_r},$$
 the generator of $H^0(\wedge^r \mc{V}_1\otimes \wedge^r \mc{V}_1^*)_0$.
\end{proof}

\begin{thm}\label{thm:kernel_independent_of_n}
When $r=n-k+1$, the kernel $\kappa$ of $$\Delta: H^0(\Sym^r(\mc{S^*\otimes S}))\to H^r(\wedge^r\mc{E}^*)$$ takes the same form for all $G(k+c,n+c)$, in terms of $I_1,...,I_r$.
\end{thm}
\begin{proof}
Applying Lemma \ref{lem:iso} to (\ref{CDq0}), we find that $\pi_0\circ q_0(\kappa_B) = \kappa_{B_1}$, when $B=\begin{pmatrix}
B_1 & 0\\ 0 & 0
\end{pmatrix}$.

Let $\kappa_B=\sum s^{\lambda,B}\kappa_\lambda$, $\kappa_{B_1}=\sum s^{\lambda,B_1}\kappa_\lambda$, then $s^{\lambda,B} = s^{\lambda,B_1}$.
Now let $$A_r=\{\alpha=(\alpha_1,...,\alpha_r)\in \mathbb{Z}^r_{\geq 0}|\sum_{i=1}^r \alpha_i \leq r\}$$
and $I^\alpha = \prod_{i=1}^r I_i^{\alpha_i}$.
Observe that $I_j(B) = I_j(B_1)$, so we can write
\begin{equation}\begin{array}{l}
\displaystyle
s^{\lambda,B} = \sum_{\alpha \in A_r} s^{\lambda,n}_{\alpha} I^{\alpha},\\
\displaystyle
s^{\lambda,B_1} = \sum_{\alpha \in A_r} s^{\lambda,n-1}_{\alpha} I^{\alpha}.
\end{array}
\end{equation}

Note that this holds for arbitrary $B_1$ with $B=\begin{pmatrix}
B_1 & 0\\ 0 & 0
\end{pmatrix}$. For $I_1, I_2, ..., I_r$, we can always solve the equation
$t^r+\sum_{i=1}^{r} t^{r-i} (-1)^i I_i =0$ and get $r$ roots $t_1,...,t_r$.
Then the matrix $\text{diag}\{t_1,...,t_r\}$ has invariants $I_1,...,I_r$.
We can take $I_1, I_2, ..., I_r$ sufficiently small such that our matrix $\text{diag}\{t_1,...,t_r\}$ is not in the degenerate locus. This implies that
\begin{equation}
\displaystyle\sum_{\alpha \in A_r} s^{\lambda,n}_{\alpha} I^{\alpha}
 = \sum_{\alpha \in A_r} s^{\lambda,n-1}_{\alpha} I^{\alpha}
\end{equation}
as an equality of two holomorphic functions of variables $I_1,..,I_r$ holds on an open set. So it holds in general by the identity theorem of
holomorphic functions of several variables.

This shows that $s^{\lambda,n}_\alpha =s^{\lambda,n-1}_\alpha$ for arbitrary $n$ and finishes the proof of the theorem.
\end{proof}

For later use, we require the following
\begin{defn}\label{kappa_tilde_def}
$$\tilde{\kappa}_{(r)} \equiv \sum_{i=0}^{{\rm min}\{r,n\}} I_i \ \kappa_{(r-i)} \cdot \kappa_{(1)}^i .$$
\end{defn}

\subsection{$B=\varepsilon I$}\label{sec:b=eI} For the special case $B=\varepsilon I$, we can derive the expression of $\kappa_B$ directly.

In this case, we have a map of short exact sequences
\begin{equation}
\xymatrix{
0\ar[r]&\Omega\ar[r]\ar[d] &\mc{V}^*\otimes \mc{S} \ar[r]^{f}\ar@{=}[d] &\mc{S}^*\otimes \mc{S}\ar[d]^{\cong}_{h}\ar[r]&0\\
0\ar[r]&\mc{E}^*\ar[r]&\mc{V}^*\otimes \mc{S} \ar[r]^{f_{B}} &\mc{S}^*\otimes \mc{S}
\ar[r]&0,
}
\end{equation}
where $h$ is given by $h: \sigma^a_b \mapsto \sigma^a_b + \varepsilon (\text{tr\ } \sigma) \delta^a_b$, for any local section $\sigma^a_b$ of $\mc{S^*\otimes S}$.

\begin{thm}\label{thm:B=epsilon I case}
When $B = \varepsilon I$,
\begin{equation}
h(\kappa_{(r)})  =
\sum_{j=0}^{k+r-n-1}  \varepsilon^j {k+r-n-1 \choose j}
\kappa_{(1)}^j \tilde{\kappa}_{(r-j)} .
 \end{equation}
\end{thm}
\begin{proof}
$\kappa_{(r)}$ is the identity bundle map on $\mathrm{Sym}^r \mathcal{S}$.
For any section $\sigma$ of $\mathcal{S^*\otimes S}$, $h$ is defined by
\[
h (\sigma^a_b) = \sigma^a_b + \varepsilon (\mathrm{Tr} \sigma) \delta^a_b,
\]
where $h (\sigma^a_b)$ represents the component of the image of $\sigma$ under $h$.
More generally, given a section $T$ of $(\mathcal{S^*\otimes S})^{\otimes r}$, the tensor product of
$r$ copies of $h$ is given by $h^r = h_1 \otimes h_2 \cdots \otimes h_r$,
where $h_i$ acts only on the $i$'th factor of $(\mathcal{S^*\otimes S})^{\otimes r}$,
specifically,
\[
h_i (T^{a_1\cdots a_r}_{b_1\cdots b_r}) =
T^{a_1\cdots a_r}_{b_1\cdots b_r} + \varepsilon
T^{a_1\cdots a_{i-1} c a_{i+1} \cdots a_r}_{b_1\cdots b_{i-1} c b_{i+1} \cdots b_r} \delta^{a_i}_{b_i}.
\]
$\kappa_{(r)}$ has components $(r!)^{-1}\delta^{(a_1\cdots a_r)}_{b_1\cdots b_r}$, $a_i=1,\cdots,k, b_i=1,\cdots,k$,
where $\delta^{(a_1\cdots a_r)}_{b_1\cdots b_r}$ denotes $\delta^{(a_1}_{b_1}\delta^{a_2}_{b_2}\cdots \delta^{a_r)}_{b_r}$,
and $(\cdots)$ denotes the symmetrization of indices.
We denote $h_1 \otimes h_2 \cdots \otimes h_i$ by $h^i$ for $i=1,\cdots r$.
Thus one can compute, $h^r(\kappa_{(r)})$ has components
\begin{equation}\label{h^r step1}
\begin{split}
h^r((r!)^{-1}\delta^{(a_1\cdots a_r)}_{b_1\cdots b_r}) &=
(r!)^{-1} h^{r-1} \left( \delta^{(a_1\cdots a_r)}_{b_1\cdots b_r} + \varepsilon
\delta^{(a_1\cdots a_{r-1}c)}_{b_1\cdots b_{r-1}c} \delta^{a_r}_{b_r} \right) \\
= &(r!)^{-1} h^{r-1} \left(\delta^{(a_1\cdots a_r)}_{b_1\cdots b_r}
+\varepsilon \delta^{(a_1\cdots a_{r-1})}_{b_1\cdots b_{r-1}}\delta^c_c \delta^{a_r}_{b_r}
+\varepsilon \sum_{i=1}^{r-1} \delta^{(a_1\cdots a_{r-1})}_{b_1\cdots \hat{b}_i \cdots b_{r-1} c} \delta^c_{b_i} \delta^{a_r}_{b_r}
\right),\\
= &(r!)^{-1} h^{r-1} \left(\delta^{(a_1\cdots a_r)}_{b_1\cdots b_r} + \varepsilon (k+r-1)
\delta^{(a_1\cdots a_{r-1})}_{b_1\cdots b_{r-1}}\delta^{a_r}_{b_r} \right).
\end{split}
\end{equation}
Define $Y_s$ to have components
\[
\delta^{(a_1\cdots a_r)}_{b_1\cdots b_r}+ \sum_{t=1}^s \sum_{r-s+1 \leqslant i_1 < \cdots < i_t \leqslant r}
\varepsilon^t \frac{(k+r-1)!}{(k+r-t-1)!}
\delta^{(a_1 \cdots \hat{a}_{i_1} \hat{a}_{i_2} \cdots \hat{a}_{i_t} \cdots a_r)}_
{b_1 \cdots \hat{b}_{i_1} \hat{b}_{i_2} \cdots \hat{b}_{i_t} \cdots b_r}
\delta^{a_{i_1}}_{b_{i_1}} \cdots \delta^{a_{i_t}}_{b_{i_t}}.
\]
We claim that
\begin{equation}
h^r(\kappa_{(r)}) = (r!)^{-1} h^{r-s} (Y_s),
\end{equation}
for $s=1,2,\cdots,r$.
This is true for $s=1$ due to \eqref{h^r step1}.
Let's assume the claim is true for some $s$, and prove that it is also true for $s+1$.
This can be shown through a direct computation as follows:
\[
\begin{split}
h^{r-s} ((Y_s)^{a_1\cdots a_r}_{b_1\cdots b_r})& \\
=&h^{r-s-1} \Biggl((Y_s)^{a_1\cdots a_r}_{b_1\cdots b_r} +
\\
& \sum_{t=0}^{s} \sum_{r-s+1 \leqslant i_1 < \cdots < i_t \leqslant r} \varepsilon^{t+1}
\frac{(k+r-1)!}{(k+r-t-2)!} \delta^{(a_1 \cdots \hat{a}_{r-s} \hat{a}_{i_1} \hat{a}_{i_2} \cdots \hat{a}_{i_t} \cdots a_r)}
_{b_1 \cdots \hat{b}_{r-s} \hat{b}_{i_1} \hat{b}_{i_2} \cdots \hat{b}_{i_t} \cdots b_r}
\delta^{a_{r-s}}_{b_{r-s}} \delta^{a_{i_1}}_{b_{i_1}} \cdots \delta^{a_{i_t}}_{b_{i_t}}
\Biggr), \\
= &h^{r-s-1}((Y_{s+1})^{a_1\cdots a_r}_{b_1\cdots b_r}).
\end{split}
\]
Thus we can take $s=r$ to have
\begin{equation}\label{h^r stepr}
h^r(\kappa_{(r)}) = (r!)^{-1} Y_r.
\end{equation}
Because $(r-t)^{-1}\delta^{(a_1 \cdots \hat{a}_{i_1} \hat{a}_{i_2} \cdots \hat{a}_{i_t} \cdots a_r)}_
{b_1 \cdots \hat{b}_{i_1} \hat{b}_{i_2} \cdots \hat{b}_{i_t} \cdots b_r}$ are the
components of $\kappa_{(r-t)}$, \eqref{h^r stepr} can be written as
\begin{equation}
\begin{split}
h^r(\kappa_{(r)})& \\
=&\kappa_{(r)} + \sum_{t=1}^{r} \sum_{1 \leqslant i_1 < \cdots < i_t \leqslant r}
\varepsilon^t \frac{(k+r-1)!(r-t)!}{(k+r-t-1)!r!} \kappa_{(r-t)} \kappa_{(1)}^t,\\
=&\kappa_{(r)} + \sum_{t=1}^{r} \varepsilon^t  {r \choose t}
\frac{(k+r-1)!(r-t)!}{(k+r-t-1)!r!} \kappa_{(r-t)} \kappa_{(1)}^t,\\
=&\sum_{t=0}^r {k+r-1 \choose t} \kappa_{(r-t)} \kappa_{(1)}^t \varepsilon^t.
\end{split}
\end{equation}

With the aid of the combinatorial formula
\[
{m+n \choose l} = \sum_{i=0}^{m} {m \choose i}{n \choose l-i},
\]
where ${n \choose i}=0$ when $i < 0$ or $i>n$, one can compute, for $r > n-k$,
\begin{equation}
\begin{split}
h(\kappa_{(r)}) &= \sum_{i=0}^r \varepsilon^i {k+r-1 \choose i} \kappa_{(r-i)} \kappa_{(1)}^i\\
&= \sum_{i=0}^r \sum_{j=0}^{k+r-n-1} \varepsilon^j \varepsilon^{i-j} {k+r-n-1 \choose j}{n \choose i-j} \kappa_{(r-i)} \kappa_{(1)}^i\\
&= \sum_{j=0}^{k+r-n-1} \sum_{i=j}^{\min\{n+j,r\}} \varepsilon^j {k+r-n-1 \choose j} I_{i-j} \kappa_{(r-i)} \kappa_{(1)}^i\\
&= \sum_{j=0}^{k+r-n-1}  \varepsilon^j {k+r-n-1 \choose j}
\left( \sum_{i=0}^{\min\{n,r-j\}} I_i \kappa_{(r-j-i)} \kappa_{(1)}^i
\right) \kappa_{(1)}^j\\
&= \sum_{j=0}^{k+r-n-1}  \varepsilon^j {k+r-n-1 \choose j}
\kappa_{(1)}^j \tilde{\kappa}_{(r-j)}.
\end{split}
\end{equation}

\end{proof}

\subsection{The result for the $r=n-k+1$ case}

We then have an algorithm to compute the kernel for $r=n-k+1$.

From Theorem \ref{B_ij-degree}, \ref{similarity} and \ref{glV_invariance_r=n-k+1} we know that the kernel $\kappa$ is of the form
$$
\kappa = \sum s^{\lambda} \kappa_\lambda
$$
where $s^{\lambda}$ is a polynomial in $I_i$,  $i=1,...,r$,
the similarity invariants of the characteristic polynomial of $B$. Moreover Theorem \ref{B_ij-degree} and \ref{glV_invariance_r=n-k+1} guarantee that the degree of the polynomial is no more than $r$.

From Theorem \ref{thm:kernel_independent_of_n} we know that $s^{\lambda}$ has the form
\begin{equation}\label{abstract_s_lambda}
s^{\lambda} = \sum_{\alpha \in A_r} s^{\lambda}_{\alpha} I^{\alpha},
\end{equation}
where $ s^{\lambda}_{\alpha}, \alpha \in A_r $, are independent of $n$.


To determine $s^{\lambda}_\alpha$, we consider specific choices of $B$ on $G(k,n)$ with $k\geq r$.
We take $k\geq r$ because this is the `stable-range'. Namely, when $k<r$, some $\kappa_\lambda$ might be 0 for dimension reason, hence cannot be seen in the kernel relation even if they are there for $k\geq r$.

We first work out the general $h$ maps making the following diagram commute:
{
\newcommand{\A}{\mc{V^*\otimes S}}
\newcommand{\B}{\mc{S^*\otimes S}}
\newcommand{\nA}{\mc{V^*\otimes S}}
\newcommand{\nB}{\mc{S^*\otimes S}}

\begin{equation}\label{general_h_diagram}
\xymatrix{
\A      \ar[r]^{f_B }\ar@{=}[d]         &       \B  \ar[d]^{h}
\\
\nA     \ar[r]^{f_{\tilde{B}} \ }   &       \nB\ .
}
\end{equation}
}

\begin{lem}\label{general_h_iso}
If $\tilde{B}=(1+k \varepsilon) B + \varepsilon I, 1+k\varepsilon \neq 0$, then diagram \eqref{general_h_diagram} commutes.
\end{lem}
\begin{proof}
Since Hom$(\mathcal{S}^*\otimes\mathcal{S},\mathcal{S}^*\otimes\mathcal{S}) \cong H^0(\Sym^2(\mathcal{S}^*\otimes\mathcal{S})\oplus \wedge^2(\mathcal{S}^*\otimes\mathcal{S}))\cong H^0(\Sym^2(\mathcal{S}^*\otimes\mathcal{S})) \cong \mathbb{C}^2$,
we have two parameters and a general map $h$ can be written as
$\sigma^a_b \mapsto k_1\sigma^a_b + k_2 (\text{tr\ } \sigma) \delta^a_b$, where $\delta^a_b$ is the Kronecker delta function. It is an isomorphism when $k_1(k_1+k_2 k) \neq 0$, and the inverse is $(k^\prime_1,k^\prime_2) = ( \frac{1}{k_1}, -\frac{k_2}{k_1(k_1+k_2k)})$.

Writing the condition $h\circ f_B = f_{\tilde{B}}$ in coordinates, we have

\begin{equation}
k_1(c_i^a v^i_b + c_i^d B^i_j v^j_d \delta_b^a) + k_2 \delta_b^a (c_i^d v^i_d + k c_i^d B^i_j v^j_d \delta_b^a) = c_i^a v^i_b + c_i^d \tilde{B}^i_j v^j_d \delta_b^a.
\end{equation}
Take $a\neq b$, we find $k_1 =1$. Take $a=b$, we have $(1+k k_2) c_i^d B^i_j v^j_d + k_2 c_i^d v^i_d = c_i^d \tilde{B}^i_j v^j_d$, i.e.
$\tilde{B}^i_j =  (1+k k_2) B^i_j + k_2\delta^i_j$.
\end{proof}
This is the second transformation on $B$ that produces isomorphic vector bundles,
in addition to the similarity transformation. We will refer this as $\varepsilon$-Transformation, and write
\begin{equation}
\tilde{B}^i_j = ET(B)^i_j=  (1+k \varepsilon) B^i_j + \varepsilon\delta^i_j, 1+k\varepsilon \neq 0.
\end{equation}

Now let's see how to use this lemma and results of Section \ref{sec:b=eI} to determine the general form of the kernel.

\begin{thm}\label{n-k+1_generator}
For a generic deformed tangent bundle $\mc{E}$ over $X=G(k,n)$,
when $r=n-k+1$, the kernel of $H^0(\Sym^r\mc{(S^*\otimes S)}) \to $
$H^r(\wedge^r\mc{E}^*)$ is generated by
\begin{equation}\label{kernel_n-k+1}
 \tilde{\kappa}_{(r)} =
\sum_{i=0}^{r} I_i \ \kappa_{(r-i)} \cdot \kappa_{(1)}^i.
 \end{equation}
\end{thm}

\begin{proof}
Theorem~\ref{thm:B=epsilon I case} and diagram~(\ref{eq:trivb-diagram}) tell us that, on $G(k,n)$,
when $B = \varepsilon_1 I, \varepsilon_1 \neq - \frac{1}{k} $, the kernel is spanned by
\[
\tilde{\kappa}_{(r),B} =
\sum_{i=0}^{r} I_i(B) \ \kappa_{(r-i)} \cdot \kappa_{(1)}^i.
\]
Thus, by Lemma \ref{lem:iso} and (\ref{CDq0}), we see, when $B'=\begin{pmatrix}
B & 0\\ 0 & 0
\end{pmatrix}$, the kernel is spanned by
\[
\tilde{\kappa}_{(r),B} =
\sum_{i=0}^{r} I_i(B) \ \kappa_{(r-i)} \cdot \kappa_{(1)}^i\ = \sum_{i=0}^{r} I_i(B') \ \kappa_{(r-i)} \cdot \kappa_{(1)}^i\
\]
on $G(k+1,n+1)$.
Lemma \ref{general_h_iso} shows that $h(\tilde{\kappa}_{(r),B})$ is in the kernel for
\[
B''=\begin{pmatrix}
(1+(k+1)\varepsilon_2)B + \varepsilon_2 I & 0\\ 0 & \varepsilon_2
\end{pmatrix}
\]
on $G(k+1,n+1)$.
Note that
\begin{equation}
\begin{split}
\det (t I + B'') &= (t+\varepsilon_2) \det ((t+\varepsilon_2) I + (1+(k+1)\varepsilon_2)B),\\
&= \sum_{i=0}^n (t+\varepsilon_2)^{n+1-i} I_i(B) (1+(k+1)\varepsilon_2)^i,\\
&= \sum_{m=0}^{n+1} t^{n+1-m} \sum_{i=0}^{n} \varepsilon_2^{m-i} {n+1-i \choose n+1-m} I_i(B) (1+(k+1)\varepsilon_2)^i, \\
\end{split}
\end{equation}
hence
\[
I_m(B'') = \sum_{i=0}^{n} \varepsilon_2^{m-i} {n+1-i \choose n+1-m} I_i(B) (1+(k+1)\varepsilon_2)^i.
\]
From Theorem \ref{thm:B=epsilon I case}, one can compute
\[
\begin{split}
h(\tilde{\kappa}_{(r),B}) &= \sum_{i=0}^r I_i(B) h(\kappa_{(1)})^i h(\kappa_{(r-i)}),\\
&= \sum_{i=0}^r I_i(B) (1+(k+1)\varepsilon_2)^i \kappa_{(1)}^i \sum_{j=0}^{r-i} \varepsilon_2^j {n+1-i \choose j} \kappa_{(r-i-j)} \kappa_{(1)}^j,\\
&= \sum_{m=0}^{r} \kappa_{(1)}^m \kappa_{(r-m)} \sum_{i+j=m} \varepsilon_2^j I_i(B) (1+(k+1) \varepsilon_2)^i {n+1-i \choose j}, \\
&= \sum_{m=0}^{r} \kappa_{(1)}^m \kappa_{(r-m)} I_m(B''), \\
\end{split}
\]
which shows the kernel for $B''$ has the same form as $B$, namely \eqref{kernel_n-k+1}.
The same method can be applied to $B''$ in place of $B$, and induction shows
the kernel contains $\tilde{\kappa}_{(r),\tilde{B}} =
\sum_{i=0}^{r} I_i(\tilde{B}) \ \kappa_{(r-i)} \cdot \kappa_{(1)}^i$ for
$\tilde{B} = ET^l (B), l=0,1,2...,$ where $ET^0(B)=B$ and $ET^l(B)=ET(ET^{l-1}(B))$.
In particular, if we take $\tilde{\varepsilon}_i=1+(k+i-1)\varepsilon_i$, then
\[
\begin{split}
ET^0(B) =&~ \varepsilon_1 I,\\
ET^1(B) =&~ {\text{diag}} ((\tilde{\varepsilon}_2 \varepsilon_1 + \varepsilon_2)I, \varepsilon_2),\\
ET^2(B) =&~ {\text{diag}} ((\tilde{\varepsilon}_3 \tilde{\varepsilon}_2 \varepsilon_1 + \tilde{\varepsilon}_3 \varepsilon_2 + \varepsilon_3)I,
\tilde{\varepsilon}_3 \varepsilon_2 + \varepsilon_3, \varepsilon_3),\\
&~\vdots\\
ET^r(B) =&~ {\text{diag}} ((\varepsilon_{r+1}+\tilde{\varepsilon}_{r+1} \varepsilon_r +\cdots+ \tilde{\varepsilon}_{r+1}\cdots\tilde{\varepsilon}_2 \varepsilon_1)I,\\
&~\varepsilon_{r+1}+\tilde{\varepsilon}_{r+1} \varepsilon_r +\cdots+ \tilde{\varepsilon}_{r+1}\cdots\tilde{\varepsilon}_3 \varepsilon_2,
\cdots, \varepsilon_{r+1}),
\end{split}
\]
where the parameters $\varepsilon_1, \cdots, \varepsilon_{r+1}$ are such that all the
matrices above are not on the degenerate locus.
For any $\xi_1, \cdots, \xi_r \in \mathbb{C}$ such that $0< \left| \xi_i \right| \ll 1$ and $\xi_i \neq \xi_j$ for $i \neq j$,
there is a unique set of solutions to

\begin{equation}
\left\{ \begin{aligned}
         &\varepsilon_{r+1} = \xi_1, \\
         &\varepsilon_{r+1} + \tilde{\varepsilon}_{r+1} \varepsilon_r = \xi_2, \\
         &~~~\cdots  \\
         &\varepsilon_{r+1}+\tilde{\varepsilon}_{r+1} \varepsilon_r +\cdots+ \tilde{\varepsilon}_{r+1}\cdots\tilde{\varepsilon}_3 \varepsilon_2 = \xi_r,\\
         &\varepsilon_{r+1}+\tilde{\varepsilon}_{r+1} \varepsilon_r +\cdots+ \tilde{\varepsilon}_{r+1}\cdots\tilde{\varepsilon}_2 \varepsilon_1 = 0.
                          \end{aligned} \right.
                          \end{equation}
This implies the expression \eqref{kernel_n-k+1} is in the kernel for any deformation given by
\[
B = \text{diag}(0,0,\cdots, 0, \xi_r, \cdots, \xi_1)
\]
with $0< \left| \xi_i \right| \ll 1$ and $\xi_i \neq \xi_j$ for $i \neq j$.
Since this means the expression of the kernel is given by \eqref{kernel_n-k+1} for
all $I_1,I_2,\cdots,I_r$ in a small open set in $\mathbb{C}^r$, we see the kernel at order $n-k+1$
is generated by
\[
\tilde{\kappa}_{(r)} = \sum_{\lambda, \alpha} s^{\lambda}_{\alpha} I^\alpha \kappa_{\lambda} = \sum_{i=0}^{r} I_i \ \kappa_{(r-i)} \cdot \kappa_{(1)}^i
\]
for a generic deformation.
\end{proof}

Here are some examples.
For $G(n-1,n), r=2$, the result is
\begin{equation}
\kappa=(1+I_1+I_2) \kappa_{(2)}+(I_1+ I_2) \kappa_{(1,1)}.
\end{equation}

For $G(n-2,n), r=3$, the result is
\begin{equation}
\kappa=(1+I_1+I_2+I_3) \kappa_{(3)}+(I_1+2 I_2+2 I_3) \kappa_{(2,1)}
+(I_2+I_3) \kappa_{(1,1,1)}.
\end{equation}

\subsection{General $r$}
Now let's determine elements in the kernel of the connecting map for higher orders.
Let $V=V_1\oplus L$ be an $n$ dimensional vector space. Consider the inclusion of Grassmannians $X=G(k,V_1)\hookrightarrow Y = G(k,V)$ induced by $V_1\hookrightarrow V$, with $[S]\mapsto [S]$, where $S\subset V_1$ is a subspace. Note that in this case we have $\mc{V}|_X = \mc{V}_1\oplus \mc{L}$, $\mc{S}|_X=\mc{S}$, and similarly for their duals.

\begin{lem}\label{lem:CD_r-1}
Let $B=\begin{pmatrix}
B_1 & *\\ 0 & \varepsilon
\end{pmatrix}$.
Then there is a commutative diagram:
{
\newcommand{\OA}{\mc{V^*\otimes S}}
\newcommand{\OB}{\mc{S^*\otimes S}}
\newcommand{\A}{\mc{V^*\otimes S}|_X}
\newcommand{\B}{\mc{S^*\otimes S}|_X}
\newcommand{\nA}{\mc{V}_1^*\mc{\otimes S}}
\newcommand{\nB}{\mc{S}^*\otimes \mc{S}}
\begin{equation}
\xymatrix{
\OA     \ar[r]^{f_B }   \ar[d]          &       \OB \ar[d]
\\
\A      \ar[r]^{f_B }   \ar[d]^{\pi}            &       \B  \ar[d]^{\pi}
\\
\nA     \ar[r]^{f_{B_1}}                &       \nB
,}
\end{equation}
}
given by the natural projections as vertical maps.
\end{lem}

\begin{proof}This is entirely analogous to Lemma \ref{CD_gk-1n-1}.
The first square is obviously commutative. For the second one, we take the standard basis for $V=V_1\oplus L$, with $V_1=\langle e_1,...,e_{n-1}\rangle$ and $L=\langle e_n\rangle$. On $X$, each $S$ is generated by $v_1,...,v_k$ with $v^n_b = 0$ for $b\leq k$.

As before, we know the map $f_B$ and $f_{B_1}$ explicitly. \begin{equation}
f_B: c^a_i \mapsto c^a_i v^i_b +c^d_i B^i_j v^j_d \delta^a_b,
\end{equation}
and similarly for $f_{B_1}$ with $a,i$ indices runs to $k,n-1$ instead of $k,n$. Hence
\begin{equation}
\begin{array}{ll}
\pi \circ f_B - f_{B_1}\circ\pi
&=c^a_n v^n_b + \sum_{j=1}^{n-1} c^d_n B^n_j v^j_d \delta^a_b + \sum_{i=1}^{n}c^d_i B^i_n v^n_d \delta^a_b ,
\\
&=0.
\end{array}
\end{equation}
\end{proof}
\begin{thm}\label{general r kernel}
On $G(k,n)$, with $B$ outside the degenerate locus, we have
$\tilde{\kappa}_{(r)} \in \mathbb{K}_r$ for any $r\geq n-k+1$.
\end{thm}
\begin{proof}
To specify the dependence of the underlying variety $G(k,n)$ and that of the map $f_B$, we write $\mathbb{K}_r$ as $\mathbb{K}_r(k,n,B)$. Namely, $\mathbb{K}_r(k,n,B)$ is the kernel of
$$H^0(G(k,n),\Sym^r(\mc{S^*\otimes S}))\to H^r(G(k,n),\wedge^r\mc{E}^*)$$ for $\mc{E}^*$ defined as the kernel of $f_B$ in 
\begin{equation}\label{similarityCD}
\xymatrix{
\mc{V}^*\otimes \mc{S} \ar[r]^{f_B}\ar[d]^{\cong}_{g} &\mc{S}^*\otimes \mc{S}\ar[d]^{g}_{\cong}\\
\mc{V}^*\otimes \mc{S} \ar[r]^{f_{gBg^{-1}}} &\mc{S}^*\otimes \mc{S}
.}
\end{equation}

We compare $\mathbb{K}_r(k,n,B)$ with $\mathbb{K}_r(k,n-1,B_1)$. The idea here is similar to that of Theorem \ref{thm:kernel_independent_of_n}.
As a corollary of Lemma \ref{lem:CD_r-1}, we have
{
\newcommand{\C}{\mathbb{K}_r(k,n,B)}
\newcommand{\D}{\mathbb{K}_r(k,n-1,B_1)}
\newcommand{\A}{H^0(\Sym^r(\mc{S^*\otimes S}))}
\newcommand{\B}{H^r(\wedge^r\mc{E}^*)}
\newcommand{\nA}{H^0(\Sym^r(\mc{S^*\otimes S}))}
\newcommand{\nB}{H^r(\wedge^r\mc{E}^*_1)}
\begin{equation}\label{imbedding}
\xymatrix{
0 \ar[r] &\C \ar[r]\ar@{^{(}->}[d] &\A      \ar[r]  \ar@{=}[d]          &       \B  \ar[d]
\\
0 \ar[r] &\D \ar[r] &\nA        \ar[r]          &       \nB
.}
\end{equation}
}

Taking $r = n-k+1$, we have that the image of $\tilde{\kappa}_{(r)}
\in \mathbb{K}_r(k,n,B) \subset \mathbb{K}_r(k,n-1,B_1)$ is of the form
\begin{displaymath}
\imath\left( \tilde{\kappa}_{(n-k+1)} \right) =
\tilde{\kappa}_{(n-k+1)} + \varepsilon \tilde{\kappa}_{(n-k)}
\kappa_{(1)}.
\end{displaymath}
We know $\imath\left( \tilde{\kappa}_{(n-k+1)} \right) \in
 \mathbb{K}_r(k,n-1,B_1)$ by \eqref{imbedding}, and from Theorem~\ref{n-k+1_generator}
 we know
the second term $\varepsilon \tilde{\kappa}_{(n-k)}
\kappa_{(1)}$ is also in the same kernel, hence
$\tilde{\kappa}_{(n-k+1)}$ must be in the kernel.
Up to a change of notation, this shows that on $G(k,n)$, we have
$\tilde{\kappa}_{(r)} \in \mathbb{K}_r$ for $r= n-k+2$.
Then the theorem holds in general by induction.
\end{proof}

From Remark \ref{rmk:mult}, we know that $\tilde{\kappa}_{(r)}\cdot \kappa_\mu \in \mathbb{K}_s$, for any $r\geq n-k+1$ and $|\mu|=s-r$. For generic deformations of the tangent bundles, they generate the kernel $\mathbb{K}_s$.

\begin{thm}\label{thm:full_module_description}
For generic deformed tangent bundles, the kernel $\mathbb{K}_s$ is generated by $\tilde{\kappa}_{(r)}\cdot \kappa_\mu$, $r\geq n-k+1$ and $|\mu|=s-r$.
\end{thm}
\begin{proof}
Since linear independence is an open condition, it suffices to show that for the tangent bundle ($B=0$), $\mathbb{K}_s$ is generated by ${\kappa}_{r}\cdot \kappa_\mu$, $r\geq n-k+1$ and $|\mu|=s-r$.

Let $\lambda=(\lambda_1,\cdots,\lambda_k)$. Notice that the Giambelli formula for Schur functions applies to $\kappa_\lambda \in H^0(\Sym^r(\mc{S^*\otimes S}))$. In general it says
\begin{equation}
\kappa_\lambda = \det \left(
\begin{array}{cccc}
\kappa_{(\lambda_1)} & \kappa_{(\lambda_1+1)} & \cdots & \kappa_{(\lambda_1+k-1)} \\
\kappa_{(\lambda_2-1)} & \kappa_{(\lambda_2)} & \cdots & \kappa_{(\lambda_2+k-2)} \\
 & \multicolumn{2}{c}{\dotfill}\\
\kappa_{(\lambda_k-k+1)} & \kappa_{(\lambda_k-k+2)} & \cdots & \kappa_{(\lambda_k)}
\end{array} \right).
\end{equation}

Recall that for the tangent bundle, $\mathbb{K}_s$ is generated by $\kappa_\lambda$, with $\lambda=(\lambda_1,...,\lambda_k)$, $|\lambda|=s$ and $\lambda_1\geq n-k+1$.
Each such $\kappa_\lambda$ is of the form $\sum\kappa_{(r)}\cdot \kappa_\mu$
by the Giambelli formula, with $r\geq n-k+1$.

It remains to show that $h^s(\Omega^s) = h^s(\wedge^s {\mathcal
E}^*)$ for a generic deformation. Applying
Lemma~\ref{lemma:tech-app}, we see that the elements in the kernel
correspond to $H^{j-1}(\wedge^j (\mc{V^*\otimes S})\otimes
\Sym^{r-j} (\mc{S^*\otimes S}))$ with $1\leqslant j\leqslant s$,
thus we have
\begin{displaymath}
h^s(\Omega^s) \leq h^s(\wedge^s {\mathcal E}^*).
\end{displaymath}
From semicontinuity, for a generic deformation,
\begin{displaymath}
h^s(\Omega^s) \geq h^s(\wedge^s {\mathcal E}^*),
\end{displaymath}
and hence $h^s(\Omega^s) = h^s(\wedge^s {\mathcal
E}^*)$.  An alternative derivation of this result is as follows.
Note that when $q \neq p$, by semicontinuity,
\begin{displaymath}
h^q(\wedge^p {\mathcal E}^*) \leq h^q(\Omega^p) = 0,
\end{displaymath}
hence
\begin{displaymath}
h^q(\wedge^p {\mathcal E}^*) = h^q(\Omega^p) = 0.
\end{displaymath}
Since the holomorphic Euler characteristic of $\wedge^p
{\mathcal E}^*$ does not change across the flat family, we have that
$h^s(\Omega^s) = h^s(\wedge^s {\mathcal
E}^*)$ for generic deformations.
This completes the proof.
\end{proof}

This is a full description of the graded module structure of the polymology. Moreover, it implies the following description of the polymology ring.
Combining theorems \ref{quotient ring thm}, \ref{general r kernel} and \ref{thm:full_module_description}, we have
\begin{thm}\label{result}
For generic deformed tangent bundles, the polymology ring is
the ring of symmetric polynomials in $k$ indeterminates
modulo the ideal generated by the $\tilde{\kappa}$'s,
which can be given explicitly as
\begin{equation}
{\mathbb C}[\kappa_{(1)}, \kappa_{(2)}, \cdots] / \langle
D_{k+1}, D_{k+2}, \cdots, \tilde{\kappa}_{(n-k+1)},
\tilde{\kappa}_{(n-k+2)}, \cdots
\rangle,
\end{equation}
where
\begin{equation}
D_m = \det\left( \kappa_{(1+j-i)} \right)_{1 \leq i, j \leq m},
\end{equation}
and $\tilde{\kappa}_{r}$ is defined in Definition \ref{kappa_tilde_def}.
\end{thm}

\section{Non-generic situations}\label{sec:non-generic}

We now briefly discuss the non-generic situation. As mentioned in
section \ref{sec:quotient_commutativity}, the cohomology jump loci form a
subvariety of the $B$-parameter space, which we denote as
$$\mc{B}_{\text{jump}} = \cup \mc{B}^{p,q}.$$
(Note that the cohomology
$H^\bullet(\wedge^r\mc{E}^*)$ does not jump for $r\leq n-k$.)
On the other hand, the description in
Theorem~\ref{thm:full_module_description} could also break down in
non-generic cases.
Hence, we can define another subvariety of the $B$-parameter space.

Define
\begin{equation}
\tilde{\kappa}_\lambda = \det \left(
\begin{array}{cccc}
\tilde{\kappa}_{(\lambda_1)} & \tilde{\kappa}_{(\lambda_1+1)} & \cdots & \tilde{\kappa}_{(\lambda_1+k-1)} \\
\kappa_{(\lambda_2-1)} & \kappa_{(\lambda_2)} & \cdots & \kappa_{(\lambda_2+k-2)} \\
 & \multicolumn{2}{c}{\dotfill}\\
\kappa_{(\lambda_k-k+1)} & \kappa_{(\lambda_k-k+2)} & \cdots & \kappa_{(\lambda_k)}
\end{array} \right).
\end{equation}

Let $\mathscr{V}_m$ be the locus in the $B$-parameter space where
$\{\tilde{\kappa}_\lambda: |\lambda|=m, \lambda_1 \geq n-k\}$ is a linearly
dependent set. Both $\mathscr{V}_m$ and $\mc{B}_{\text{jump}}$ are
codimension at least one.  It is unclear how
$\mathscr{V}_m$ and $\cup_{p+q=2m}\mc{B}^{p,q}$ are related.
Detailed examples are given in \cite{Guo2017};
however, we leave a precise
understanding of the relationship to future work.

\section{Conjectures on quantum corrections}\label{sec:qsc}

In the companion paper
\cite{Guo2017}, physics methods were used to extract
both the classical and quantum sheaf cohomology rings for Grassmannians
with deformations of the tangent bundle.
Briefly, it was argued there that the quantum sheaf cohomology ring
could be written as
\begin{equation}\label{conjecture}
\begin{array}{l}
{\mathbb C}\left[\kappa_{(1)}, \kappa_{(2)}, \cdots  \right] / \left\langle
D_{k+1}, D_{k+2}, \cdots, \tilde{\kappa}_{(n-k+1)}, \cdots,
\tilde{\kappa}_{(n-1)},
\right. \\
\hspace*{1.5in} \left. \tilde{\kappa}_{(n)}+q,
\tilde{\kappa}_{(n+1)} + q \kappa_{(1)},
\tilde{\kappa}_{(n+2)} + q \kappa_{(2)}, \cdots
 \right\rangle,
\end{array}
\end{equation}
and it was also shown how this reduces classically to both the classical
sheaf cohomology ring above in the special case that $q \rightarrow 0$,
and to the ordinary quantum cohomology ring in the special case that
$\mc{E}=T_X$.

We have not yet demonstrated the result above for QSC mathematically, hence
we state the physics result above as a conjecture, left for future
work.

\section{Conclusions}\label{sec:conclusion}

In this paper we have proven results on the classical sheaf cohomology
ring (polymology) of Grassmannians with deformations of the tangent
bundle (Theorem \ref{result}), the first step towards mathematically
proving the form of the quantum sheaf
cohomology (QSC) ring \eqref{conjecture} arrived at via physics methods in
the companion paper \cite{Guo2017}.

\section{Acknowledgements}

We would like to thank R.~Donagi, S.~Katz, I.~Melnikov, and L.~Mihalcea
for useful discussions.  Z.L. was partially supported by
EPSRC grant EP/J010790/1.  E.S. was partially supported by
NSF grant PHY-1417410.

\appendix

\section{A technical result}
\label{app:tech}

In this appendix we will establish a technical result which will
be applied in the main text.

In general, given an exact sequence of holomorphic vector bundles over a complex manifold~$X$~as follows,
\begin{equation}
0 \to E_0 \stackrel{i_0}{\to} E_1 \stackrel{f_1}{\to} E_2 \stackrel{f_2}{\to} \cdots \stackrel{f_{n-1}}{\to}
E_n \stackrel{f_n}{\to} E_{n+1} \to 0
\end{equation}
we can decompose it into a series
of short exact sequences
\begin{equation} \label{multises}
\begin{split}
 0 \to E_0 \stackrel{i_0}{\to} &E_1 \stackrel{f_1}{\to} S_1 \to 0 \\
 0 \to S_1 \stackrel{i_1}{\to} &E_2 \stackrel{f_2}{\to} S_2 \to 0 \\
 &\vdots \\
 0 \to S_{n-2} \stackrel{i_{n-2}}{\to} &E_{n-1} \stackrel{f_{n-1}}{\to} S_{n-1} \to 0 \\
 0 \to S_{n-1} \stackrel{i_{n-1}}{\to} &E_n \stackrel{f_n}{\to} E_{n+1} \to 0.
 \end{split}
 \end{equation}
We have a long exact sequence of Cech cohomology associated with
each short exact sequence. Let's denote by~$C^q(X,E)$~the group of
$q$-cochains of $E$ corresponding to some open cover of $X$, also by
$d$ the coboundary map. Given~$\alpha \in H^0(X, E_{n+1})$, and its
representative~$\sigma_0 \in C^0(X,E_{n+1})$~with~$d \sigma_0 = 0$,
there exists a~$\tau_0 \in C^0(X,E_n)$, such
that~$f_n(\tau_0)=\sigma_0$. In turn, we have a~$\sigma_1 \in
C^0(X,S_{n-1})$, such that~$i_{n-1}(\sigma_1)=d \tau_0$~with $d
\sigma_1 = 0$. The connecting map~$\phi_0:H^0(X,E_{n+1}) \to
H^1(X,S_{n-1})$~is defined by $\phi_0(\alpha) = \phi_0([\sigma_0]) =
[\sigma_1]$. Similarly, for
\[ 0 \to S_{m} \stackrel{i_{m}}{\to} E_{m+1} \stackrel{f_{m+1}}{\to} S_{m+1} \to 0 \]
with~$m = 1,2,\cdots,n-2$, we
have~$\phi_{n-m-1}:H^{n-m-1}(X,S_{m+1}) \to H^{n-m}(X,S_m)$~defined
by $\phi_{n-m-1}([\sigma_{n-m-1}]) = [\sigma_{n-m}]$, where~$d
\sigma_{n-m-1} = d \sigma_{n-m} = 0$, $f_{m+1}(\tau_{n-m-1}) =
\sigma_{n-m-1}, i_m(\sigma_{n-m}) = d \tau_{n-m-1}$. For
\[ 0 \to E_0 \stackrel{i_0}{\to} E_1 \stackrel{f_1}{\to} S_1 \to 0,\]
$\phi_{n-1}:H^{n-1}(X,S_1) \to H^{n}(X,E_0)$~ is defined
by~$\phi([\sigma_{n-1}]) = [\sigma_n]$, where~$d \sigma_{n-1} = d
\sigma_{n} = 0$ and~$f_{1}(\tau_{n-1}) = \sigma_{n-1},
i_0(\sigma_{n}) = d \tau_{n-1}$. The map~$\delta_n : H^0(X,E_{n+1})
\to H^n(X,E_0)$~is defined to be~$\phi_{n-1}\circ \phi_{n-2} \circ
\cdots \circ \phi_1 \circ \phi_0$. We can prove the following

\begin{lem}    \label{lemma:tech-app}
Given a class~$\alpha$~in~$H^0(X,E_{n+1})$,~$\delta_n(\alpha) =
0$~if and only if, for any representative of~$\alpha$,~denoted by
$\sigma_0$, there exists an integer~$s$, $0\leqslant s \leqslant
n-1$, and a sequence~$\sigma_0, \sigma_1, \cdots, \sigma_s$ defined
above, such that~$\sigma_s$~admits a closed inverse
in~$C^s(X,E_{n-s})$~under~$f_{n-s}$.
\end{lem}
\begin{proof}
If~$\delta_n(\alpha)=0$~and~$\alpha=[\sigma_0]$, then~$\phi_{n-1}\circ \phi_{n-2} \circ
\cdots \circ \phi_1 \circ \phi_0 ([\sigma_0])=0$~and there is an integer~$s$,$0\leqslant s \leqslant n-1$,
such that~$\phi([\sigma_s])=[\sigma_{s+1}]=0$. From the short exact sequence~$0 \to S_{n-s-1} \stackrel{i_{n-s-1}}{\to}
E_{n-s} \stackrel{f_{n-s}}{\to} S_{n-s} \to 0$, we have the long exact sequence
\[
\cdots \to H^s(X,E_{n-s})\stackrel{f_{n-s}}{\to} H^s(X,S_{n-s})\stackrel{\phi_s}{\to}
H^{s+1}(X,S_{n-s-1}) \to \cdots
\]
thus~$[\sigma_s]\in \ker \phi_s = \mathrm{Im} f_{n-s}$, then there
exists~$[\omega]\in H^s(X,E_{n-s})$, such
that~$f_{n-s}([\omega])=[\sigma_s]$, with~$\omega \in
C^s(X,E_{n-s}), d \omega = 0$. This implies~$\sigma_s -
f_{n-s}(\omega)$~is exact, and there
exists~$\eta$~in~$C^{s-1}(X,S_{n-s})$~such that $\sigma_s =
f_{n-s}(\omega) + d \eta$. Since~$f_{n-s}$~is surjective (in the
short exact sequence), we can find $\tilde{\eta}\in
C^{s-1}(X,E_{n-s})$, such that~$f_{n-s}(\tilde{\eta})=\eta$.
Then~$\sigma = f_{n-s}(\omega+ d \tilde{\eta})$, i.e. $\sigma_s$~has
a closed inverse under~$f_{n-s}$.\par Conversely, if for
some~$0\leqslant s \leqslant n-1$, $\sigma_s$~has a closed inverse,
say~$f_{n-s}(\omega)=\sigma_s$, with~$\omega \in C^s(X,E_{n-s}), d
\omega=0$. We can take~$\tau_s=\omega$~in our definition of the
connecting map, then
\[
i_{n-s-1}(\sigma_{s+1}) = d \tau_s = d \omega = 0.
\]
But~$i_{n-s-1}$~is injective, which implies~$\sigma_{s+1}=0$, then~$\phi_s([\sigma_s])=
[\sigma_{s+1}]=0$, and~$\delta_n([\sigma_0])=\phi_{n-1}\circ \phi_{n-2} \circ
\cdots \circ \phi_s ([\sigma_s]) = 0$.
\end{proof}

\bibliography{matha2.bbl}
\bibliographystyle{plain}
\end{document}